\documentclass[a4paper, 11pt, english]{article}

\usepackage{etex}

\newcommand{\um}[1]{{\color{blue}\textit{#1}}}

\newcommand{\dneg}{{\sim}}
\newcommand{\DNEG}{{\ast}}

\usepackage{lscape}
\usepackage{pdflscape}

\usepackage{stmaryrd}

\usepackage{cmll}

\usepackage{comment}

\usepackage{amscd}
\usepackage{amssymb,amsmath,amsthm}
\usepackage{mathptmx}
\usepackage{mathrsfs}
\usepackage{color}
\usepackage{xspace}
\usepackage{bussproofs}
\EnableBpAbbreviations

\usepackage{tikz}

\usepackage{txfonts}

\usepackage{bpextra}

\usepackage{enumerate}

\usepackage{centernot}

\usepackage{mathtools}

\usepackage{hyperref}

\usepackage{cleveref}

\usepackage{relsize}

\usepackage{caption}

\usepackage{multirow}

\usepackage{multicol}

\usepackage{array}
\newcolumntype{C}[1]{>{\centering\arraybackslash}p{#1}}
\newcolumntype{L}[1]{>{\arraybackslash}p{#1}}

\usepackage[left=3cm,top=3cm,right=3cm,bottom=2cm]{geometry}
\newtheorem{theorem}{Theorem}[section]
\newtheorem{corollary}[theorem]{Corollary}
\newtheorem{lemma}[theorem]{Lemma}
\newtheorem{proposition}[theorem]{Proposition}

\newtheorem{fact}[theorem]{Fact}

\newtheorem{definition}[theorem]{Definition}

\newtheorem{remark}[theorem]{Remark}

\def\fCenter{{\mbox{$\ \vdash\ $}}}

\newcommand{\fns}{\footnotesize}

\newcommand{\marginnote}[1]{\marginpar{\raggedright\tiny{#1}}}

\def\mc{\multicolumn}

\newcommand{\band}{\ensuremath{\otimes}\xspace}
\newcommand{\bor}{\ensuremath{\oplus}\xspace}
\newcommand{\btim}{\ensuremath{\times}\xspace}
\newcommand{\bplu}{\ensuremath{+}\xspace}
\newcommand{\bt}{\ensuremath{\texttt{t}}\xspace}
\newcommand{\bfb}{\ensuremath{\texttt{f}}\xspace}
\newcommand{\bkt}{\ensuremath{\top}\xspace}
\newcommand{\bkb}{\ensuremath{\bot}\xspace}
\newcommand{\BT}{\ensuremath{\check{\texttt{t}}}\xspace}

\newcommand{\utop}{\ensuremath{1_1}\xspace}
\newcommand{\ubot}{\ensuremath{0_1}\xspace}
\newcommand{\uand}{\sqcap_1}
\newcommand{\uor}{\sqcup_1}
\newcommand{\urarr}{\ensuremath{>}\xspace}
\newcommand{\ucrarr}{\ensuremath{\,{>\mkern-7mu\raisebox{-0.065ex}{\rule[0.5865ex]{1.38ex}{0.1ex}}}\,}\xspace}
\newcommand{\UTOP}{\ensuremath{\hat{1}_1}\xspace}
\newcommand{\UBOT}{\ensuremath{\check{0}_1}\xspace}
\newcommand{\UAND}{\ensuremath{\:\hat{\sqcap}_1\:}\xspace}
\newcommand{\UOR}{\ensuremath{\:\check{\sqcup}_1\:}\xspace}
\newcommand{\URARR}{\ensuremath{\:\check{\sqsupset}_1\:}\xspace}
\newcommand{\UCRARR}{\ensuremath{\:\hat{\sqsubset}_1\:}\xspace}

\newcommand{\dtop}{\ensuremath{1_2}\xspace}
\newcommand{\dbot}{\ensuremath{0_2}\xspace}
\newcommand{\dand}{\ensuremath{\sqcap_2}\xspace}
\newcommand{\dor}{\ensuremath{\sqcup_2}\xspace}
\newcommand{\drarr}{\ensuremath{\sqsupset_2}\xspace}
\newcommand{\dcrarr}{\ensuremath{\,{\sqsubset_2}\,}\xspace}
\newcommand{\DTOP}{\ensuremath{\hat{1}_2}\xspace}
\newcommand{\DBOT}{\ensuremath{\check{0}_2}\xspace}
\newcommand{\DAND}{\ensuremath{\:\hat{\sqcap}_2\:}\xspace}
\newcommand{\DOR}{\ensuremath{\:\check{\sqcup}_2\:}\xspace}
\newcommand{\DRARR}{\ensuremath{\:\check{\sqsupset}_2\:}\xspace}
\newcommand{\DCRARR}{\ensuremath{\hat{{\:{\sqsubset}_2\:}}}\xspace}

\newcommand{\bu}{\ensuremath{{\pi}_1}\xspace}
\newcommand{\bd}{\ensuremath{{\pi}_2}\xspace}
\newcommand{\ubtu}{\ensuremath{{\Diamond}_1}\xspace}
\newcommand{\dbtu}{\ensuremath{{\vartriangleright}_2}\xspace}
\newcommand{\ubtd}{\ensuremath{{\Box}_1}\xspace}
\newcommand{\dbtd}{\ensuremath{{\vartriangleleft}_2}\xspace}
\newcommand{\ubku}{\ensuremath{\boxdot_1}\xspace}
\newcommand{\dbku}{\ensuremath{\boxdot_2}\xspace}
\newcommand{\ubkd}{\ensuremath{\Diamonddot_1}\xspace}
\newcommand{\dbkd}{\ensuremath{\Diamonddot_2}\xspace}
\newcommand{\ud}{\ensuremath{\mathrm{n}}\xspace}
\newcommand{\du}{\ensuremath{\mathrm{p}}\xspace}
\newcommand{\ubu}{\ensuremath{\ell_1}\xspace}
\newcommand{\dbu}{\ensuremath{\ell_2}\xspace}

\newcommand{\BU}{\ensuremath{{\Pi}_1}\xspace}
\newcommand{\BD}{\ensuremath{{\Pi}_2}\xspace}

\newcommand{\UD}{\ensuremath{\mathrm{N}}\xspace}
\newcommand{\DU}{\ensuremath{\mathrm{P}}\xspace}

\newcommand{\drarrk}{\ensuremath{\,{>\mkern-7mu\raisebox{-0.065ex}{\rule[0.5865ex]{1.38ex}{0.1ex}}}\,}\xspace}

\usetikzlibrary{arrows}
\usetikzlibrary{matrix}
\usetikzlibrary{patterns}
\usetikzlibrary{shapes}

\usepackage{authblk}

\makeindex

\title{Bilattice Logic Properly Displayed\thanks{This research is supported by the NWO Vidi grant 016.138.314, the NWO Aspasia grant 015.008.054, a Delft Technology Fellowship awarded to the third author in 2013, and the International Program Fund for Ph.D. candidates, Sun Yat-sen University 2016.}}

\author[1]{Giuseppe Greco}
\author[1,2]{ Fei Liang}
\author[1,3]{Alessandra Palmigiano}
\author[4]{Umberto Rivieccio}
\affil[1]{Delft University of Technology, the Netherlands}
\affil[2]{Institute of Logic and Cognition, Sun Yat-sen University, China}
\affil[3]{University of Johannesburg, South Africa}
\affil[4]{Federal University of Rio Grande do Norte, Brazil}

\date{}

\begin{document}

\maketitle

\begin{abstract}
We introduce a proper multi-type display calculus for bilattice logic (with conflation) for which we prove soundness, completeness, conservativity, standard subformula property and cut-elimination. Our proposal builds on the product representation of bilattices and applies the guidelines of the multi-type methodology in the design of display calculi. 


2010 Math.~Subj.~Class.---03F52, 03F05, 03G10.
\end{abstract}

%\tableofcontents

%DELETE THE CONTENTS
\section{Introduction}

Bilattices are algebraic structures introduced in \cite{ginsberg1988multivalued} in the context of a multivalued approach to deductive reasoning, and have subsequently found applications in a variety of areas in computer science and artificial intelligence.
The basic intuition behind 
the bilattice formalism, 
which can be traced back to the work of
Dunn and Belnap~\cite{dunn1966algebra, belnap1977computer, belnap1977useful}, is to carry out reasoning within a space of truth-values that results from expanding the classical set $\{\bfb, \bt \}$ with a value $\bot$, representing lack of information, and a value $\top$, representing over-defined or contradictory information.

During the last two decades, the theory of bilattices has been investigated in depth from a logical
and algebraic point of view:  complete (Hilbert- and Gentzen-style) presentations of bilattice-based logics   were introduced in  \cite{arieli1996reasoning, arieli1998value}, followed by \cite{bou2010logic%, bou2011varieties
} which focuses on the implication-free reduct of the logic.
The calculi introduced in these papers have many common aspects with those considered e.g.~in~\cite{font1997belnap} for the Belnap-Dunn logic, of which bilattice logics are conservative expansions.

Negation plays a very special role, and it is in fact  %within bilattice logics: it is 
due to the negation connective that
 bilattice logics are not \emph{self-extensional} \cite{wojcicki1988referential} (or, as other authors say, \emph{congruential}%\cite{arieli1996reasoning}
 ), i.e.~the inter-derivability
relation of the logic is not a congruence of the formula algebra. This means that there are formulas such that
$\varphi \dashv \vdash \psi $ and yet  $\neg \varphi \not \! \dashv \vdash \neg  \psi $ (which did not happen in the Belnap-Dunn logic that is indeed self-extensional). In the Gentzen-style calculus for bilattice logic  \emph{GBL} introduced in~\cite[Section 3.2]{arieli1996reasoning}, each binary connective is introduced via four different logical rules, two of which are standard, and introduce it as main connective on the left and on the right of the turnstyle, and two non-standard rules, which introduce the same connective under the scope of a negation. From a proof-theoretic perspective, this solution presents the disadvantage that the resulting calculus is not fully modular, does not support a proof-theoretic semantics, and does not enjoy the standard  subformula property.

In this paper we introduce a {\em proper multi-type display calculus} for bilattice logic that circumvents all the above-mentioned disadvantages\footnote{The notion of proper display calculus has been introduced in \cite{Wansing1998}. Properly displayable logics, i.e.~those which can be captured by some proper display calculus, have been characterized in a  purely proof-theoretic way in \cite{CiabattoniRamanayake2016}. In \cite{GMPTZ}, an alternative characterization of properly displayable logics was introduced which builds on the algebraic theory of unified correspondence \cite{CoGhPa14}.}. 
%mentioned above. 
The design of our calculus follows the principles of the {\em multi-type} methodology introduced in \cite{GKPLori,Multitype,PDL,TrendsXIII} with the aim of displaying dynamic epistemic logic and propositional dynamic logic
and subsequently  applied %successfully 
to several other logics (%such as
e.g.~linear logic with exponentials \cite{GP:linear}, inquisitive logic \cite{Inquisitive}, semi-De Morgan logic \cite{SDM}, lattice logic \cite{GrecoPalmigianoLatticeLogic}) which are not properly displayable in their single-type presentation, which also inspired the design of novel logics \cite{BGPTW}. Our multi-type syntactic presentation of bilattice logic
is based on the algebraic insight provided by the product representation theorems
(see e.g.~\cite{bou2011varieties}) and possesses all the %several 
desirable properties of {\em proper} display calculi. In particular, our calculus enjoys the standard subformula property, supports a proof-theoretic semantics and  is fully modular.

\paragraph{Structure of the paper}
In Section~\ref{sec:prel} we recall  basic definitions and results about bilattices and bilattice logics.
Section~\ref{sec:alg} presents an algebraic analysis of bilattices as heterogeneous structures which provides
a basis for our multi-type approach to their proof theory. %of bilattices. 
Our display calculus is introduced in Section~\ref{sec:disp} where we also prove soundness, completeness, conservativity, subformula property and cut-elimination. In Section~\ref{sec:conc} we outline some directions for future work. 

\section{Preliminaries on bilattices}
\label{sec:prel}

The following definitions and results can be found e.g.~in~\cite{arieli1996reasoning, bou2010logic}.

  \begin{definition}%[cf.~\cite{arieli1996reasoning}]
A {\em bilattice} is a structure $\mathbb{B} =$ ${(B, \leq_t, \leq_k,  \neg)}$ such that $B$ is a non-empty set%containing at least one element
, $(B, \leq_t)$, $(B, \leq_k)$ are lattices, and $\neg$ is a unary operation on $B$ having the following properties:
\begin{itemize}
\item if $a \leq_t b$, then $\neg b \leq_t \neg a$,
\item if $a \leq_k b$, then $\neg a \leq_k \neg b$,
\item  $\neg\neg a =  a$.
\end{itemize}
\end{definition}
%In the sequel of the paper, 
We use $\wedge, \vee$ for the lattice operations which correspond to $\leq_t$ and $\band, \bor$ for those that correspond to $\leq_k$.
If present, the lattice bounds of $\leq_t$ are denoted by 
$ \bfb$ and $\bt $ (minimum and maximum, respectively) and those of  $\leq_k$ by 
$\bot$ and $\top$.
The smallest non-trivial bilattice is the four-element one (called
\textbf{Four}) with universe $\{ \bfb, \bt, \bot, \top \}$.

\begin{fact}
The following equations (\emph{De Morgan laws for negation}) hold in any bilattice: 
\begin{center}
\begin{tabular}{lcl}
$\neg(x \wedge y) = \neg x \vee \neg y$,&  &$\neg(x \vee y) = \neg x \wedge \neg y $,\\
$\neg(x \band y) = \neg x \band \neg y$,&  &$\neg(x \bor y) = \neg x \bor \neg y $.\\
\end{tabular}
\end{center}
Moreover, if the bilattice is bounded, then $$\neg\bt = \bfb,\ \ \neg\bfb = \bt,\ \  \neg\top = \top,\ \ \neg\bot = \bot.$$ 
\end{fact}

% In the following, we use $\mathsf{B}$ to denote the class of all bilattices.
%
% \begin{definition}[cf.~\cite{arieli1996reasoning}]
%A bilattice is  {\em interlaced} if each  of $\{\wedge, \vee, \band, \bor\}$ is monotonic with respect to both $\leq_t$ and $\leq_k$. That is, if the following quasi-equations hold:
%\vspace{-2.5ex}
%
%\begin{center}
%\begin{tabular}{lcl}
%$x \leq_t y \Rightarrow x \band z \leq_t  y \band z$, &  & $x \leq_t y \Rightarrow x \bor z \leq_t  y \bor z$,\\
%$x \leq_k y \Rightarrow x \wedge z \leq_k  y \wedge z$, &  & $x \leq_k y \Rightarrow x \vee z \leq_k  y \vee z$. \\
%\end{tabular}
%\end{center}   
%\end{definition}
%If an  interlaced bilattice is  bounded, then $$\bt \band \bfb = \bot,\ \  \bt \bor \bfb = \top,\ \  \top \wedge \bot = \bfb,\ \  \top \vee \bot = \bt.$$

\begin{definition}%[cf.~\cite{arieli1996reasoning}]
A bilattice is called {\em distributive} when all possible distributive laws concerning the four lattice operations, i.e., all identities of the following form, hold:
\[ x \circ (y \bullet z) \thickapprox (x \circ y) \bullet (x \circ z)  \quad\, \mathrm{for} \,\, \mathrm{all} \,\, \circ,\bullet \in \{\wedge, \vee, \band, \bor\} \]
\end{definition}

If a  distributive bilattice is  bounded, then $$\bt \band \bfb = \bot,\ \  \bt \bor \bfb = \top,\ \  \top \wedge \bot = \bfb,\ \  \top \vee \bot = \bt.$$

%\begin{lemma}
%Every distributive bilattice is interlaced.
%\end{lemma}

In the following, we use $\mathsf{B}$ to denote the class of bounded distributive bilattices.

\begin{theorem}[Representation of distributive bilattices%, cf.~\cite{arieli1996reasoning}
]
\label{th:rep}
Let $\mathbb{L}$ be a bounded distributive lattice with join $\sqcup$ and meet $\sqcap$. 
Then the algebra $\mathbb{L} \odot \mathbb{L}$ 
having as universe the direct product $L \times L $
is a distributive bilattice with the following operations:

\begin{center}
\begin{tabular}{rcl}
$\langle a_1, a_2\rangle \wedge \langle b_1, b_2 \rangle$&$:=$ & $\langle a_1 \sqcap b_1,a_2\sqcup b_2 \rangle$\\
$\langle a_1, a_2\rangle \vee \langle b_1, b_2\rangle$&$:=$ & $\langle a_1 \sqcup b_1,a_2\sqcap b_2 \rangle$\\
$\langle a_1, a_2 \rangle \otimes \langle b_1, b_2 \rangle$&$:=$ & $\langle a_1 \sqcap b_1,a_2\sqcap b_2 \rangle$\\
$\langle a_1, a_2\rangle \oplus \langle b_1, b_2 \rangle$&$:=$ & $\langle a_1 \sqcup b_1,a_2\sqcup b_2 \rangle$\\
$\neg \langle a_1, a_2\rangle $&$ :=$ &$\langle a_2, a_1 \rangle$\\
$\bfb$&$:=$ & $\langle 0,  1\rangle$\\
$\bt$&$:=$ & $\langle 1,  0 \rangle$\\
$\bot$&$:=$ & $\langle 0,  0\rangle$\\
$\top$&$:=$ & $\langle1,  1\rangle$
\end{tabular}
\end{center}
\end{theorem}

\begin{theorem}
Every distributive bilattice is isomorphic to $\mathbb{L} \odot \mathbb{L}$ for some distributive lattice $\mathbb{L}$.
\end{theorem}

 \begin{definition}%[cf.~\cite{arieli1996reasoning}]
A  structure $\mathbb{B} = (B, \leq_t, \leq_k,  \neg, -)$ is a {\em bilattice with conflation} if the reduct $ (B, \leq_t, \leq_k,  \neg)$ is a bilattice and the \emph{conflation} $-: B \rightarrow B$ is an operation satisfying:
\begin{itemize}
\item if $a \leq_t b$, then $- a \leq_t - b$;
\item if $a \leq_k b$, then $- b \leq_k - a$;
\item $-- a =  a$.
\end{itemize}
We say that $\mathbb{B}$ is  {\em commutative} if it also satisfies the equation: $\neg - x = -\neg x$. 
%$\mathbb{B}$ is called {\em classical} if it also satisfies the equation: $x \vee -\neg x = \bt$. 
\end{definition}

  \begin{fact}
The following equations 
(\emph{De Morgan laws for conflation})
 hold in any bilattice with conflation:
\begin{center}
\begin{tabular}{lcl}
$- (x \wedge y) = - x \wedge - y$&  &$- (x \vee y)  = - x \vee - y $ \\
$- (x \band y)  = - x \bor - y$&  &$- (x \bor y) = - x \band - y $ \\
\end{tabular}
\end{center}  
Moreover, if the bilattice is bounded, then $$-\bt = \bt,\ \  -\bfb = \bfb,\ \  -\top = \bot,\ \  -\bot = \top.$$
\end{fact}  

%In the following, w
We denote by $\mathsf{CB}$ the class of bounded commutative  distributive bilattices with conflation.

\begin{theorem}%[cf.~\cite{arieli1996reasoning}]
Let  $\mathbb{D} = (D, \sqcap, \sqcup, \dneg, 0, 1)$ be a De Morgan algebra, then $\mathbb{D} \odot \mathbb{D}$ is a bounded commutative  distributive bilattice with conflation where:
\begin{itemize}
\item $(D, \sqcap, \sqcup, 0, 1) \odot (D, \sqcap, \sqcup, 0, 1)$ is a bounded distributive bilattice;
\item $-(a, b) = (\dneg b, \dneg a)$;
\end{itemize}
\end{theorem}

\begin{theorem}
Every bounded commutative distributive bilattice with conflation is isomorphic to $\mathbb{D} \odot \mathbb{D}$ for some De Morgan algebra 
%involutive lattice 
$\mathbb{D}$.
\end{theorem}

%\subsection*{Hilbert-style presentation}
\subsection*{A calculus for bilattice logic}
The language of bilattice logic $\mathcal{L}$ over a denumerable set $\mathsf{AtProp} = \{p, q, r,\ldots\}$ of atomic propositions is generated as follows:
$$A ::= p \mid \bt \mid \bfb \mid \bkt \mid \bkb \mid \neg A \mid A \wedge A \mid A \vee A \mid A \band A \mid A \bor A,$$
the language of bilattice logic with conflation also includes the conflation formula $- A$.

The calculus
%Hilbert-style system 
for bilattice logic %H.
BL consists of the following axioms: 
$$A \fCenter A,\quad \neg\neg A \dashv\vdash A,$$
$$\bfb \fCenter A,\ \  A \fCenter \bt,\ \  \bkb \fCenter A,\ \  A \fCenter \bkt,$$
$$A \fCenter \neg\bfb,\ \  \neg\bt \fCenter A,\ \  \neg\bkb \fCenter A,\ \  A \fCenter \neg\bkt,$$
$$A \wedge B \fCenter A,\ \  A\wedge B \fCenter B,\ \  A \fCenter A \vee B,\ \  B \fCenter A \vee B,$$
$$A \band B \fCenter A,\ \  A\band B \fCenter B,\ \  A \fCenter A \bor B,\  B \fCenter A \bor B,$$
$$A \wedge (B \vee C) \fCenter (A \wedge B) \vee (A \wedge C),$$
$$A \band (B \bor C) \fCenter (A \band B) \vee (A \bor C),$$
$$\neg(A \wedge B) \dashv\vdash \neg A \vee \neg B,\ \  \neg(A \vee B) \dashv\vdash \neg A \wedge \neg B,$$
$$\neg(A \band B) \dashv\vdash \neg A \band \neg B,\ \  \neg(A \bor B) \dashv\vdash \neg A \bor \neg B,$$
and the following rules:
\begin{center}
\begin{tabular}{cc}
\mc{2}{c}{
\AX $A \fCenter B$
\AX $B \fCenter C$
\BI  $A \fCenter C$ 
\DP}
 \\
 & \\
\AX $A \fCenter B$
\AX $A \fCenter C$
\BI  $A \fCenter B \wedge C$ 
\DP
 & 
\AX $A \fCenter B$
\AX $C \fCenter B$
\BI  $A \vee C \fCenter B$ 
\DP
 \\

 & \\

\AX $A \fCenter B$
\AX $A \fCenter C$
\BI  $A \fCenter B \band C$ 
\DP
 & 
\AX $A \fCenter B$
\AX $C \fCenter B$
\BI  $A \bor C \fCenter B$ 
\DP
 \\

\end{tabular}
\end{center}

The calculus
%Hilbert-style system 
for bilattice logic with conflation CBL consists of the axioms and rules of BL plus the following axioms:
$$-- A \dashv\vdash A, \ \ -\neg A \dashv\vdash \neg- A,$$
$$-\bfb \fCenter A,\ \  A \fCenter -\bt,\ \  -\bkt \fCenter A,\ \  A \fCenter -\bkb,$$
$$-(A \wedge B) \dashv\vdash - A \wedge - B,\ \  -(A \vee B) \dashv\vdash - A \vee - B,$$
$$-(A \band B) \dashv\vdash - A \bor - B,\ \  -(A \bor B) \dashv\vdash - A \band - B.$$
The algebraic semantics of BL (resp.~CBL) is given by $\mathsf{B}$ (resp.~$\mathsf{CB}$).  We use $A \vDash_{\mathsf{B}} C$ (resp.~$A \vDash_{\mathsf{CB}} C$) to mean: for any $\mathbb{B} \in \mathsf{B}$ (resp.~$\mathbb{B} \in \mathsf{CB}$), if 
%$A^{\mathbb{B}} \in \{\bt, \bkt\}$
$A^{\mathbb{B}} \in F_{\bt} $
then 
%$C^{\mathbb{B}} \in \{\bt, \bkt\}$, 
$C^{\mathbb{B}} \in  F_{\bt} $.
Here $A^{\mathbb{B}},  C^{\mathbb{B}}$ mean the interpretations of $A$ and $C$ in $\mathbb{B}$, respectively;
and $F_{\bt} = \{ a \in B : \bt \leq_k a  \}$ is the set of designated elements of $\mathbb{B}$ 
(using the terminology of \cite[Definition 2.13]{arieli1996reasoning},  $F_{\bt}$ is the \emph{least bifilter} of  $\mathbb{B}$).

Soundness of BL (resp.~CBL) is straightforward. In order to show  completeness, we can prove that every axiom and rule of 
Arieli and Avron's
$\mathit{GBL}$ (resp.~$\mathit{GBS}$, cf.~\cite{arieli1996reasoning}) is derivable in BL (resp.~CBL)\footnote{In 
order to do this, we view a sequent $\Gamma \Rightarrow \Delta$ of  $\mathit{GBL}$ ($\mathit{GBS}$)
as the equivalent sequent $\bigwedge\Gamma \Rightarrow \bigvee\Delta$.}.
%using the fact that in $\mathit{GBL}$ ($\mathit{GBS}$) any sequent $\Gamma \Rightarrow \Delta$
%is equivalent to the sequent $\bigwedge\Gamma \Rightarrow \bigvee\Delta$.
 %via the translation from $\Gamma \Rightarrow \Delta$ to $\bigwedge\Gamma \Rightarrow \bigvee\Delta$. 
Then the completeness of BL (resp.~CBL) follows  from the completeness of $\mathit{GBL}$ (resp.~$\mathit{GBS}$, \cite[Theorem 3.7]{arieli1996reasoning}).

 \begin{theorem}[Completeness]
\label{thm:com}
$A \vdash_\mathrm{BL} C$ iff $A \vDash_{\mathsf{B}} C$ (resp.~$A \vdash_\mathrm{CBL} C$ iff $A \vDash_{\mathsf{CB}} C$).
\end{theorem}

%It has been known that
%\begin{center}
%\begin{tabular}{lllllll}
%$A \fCenter B$&$ \mathrm{and}$&$\neg A \fCenter \neg B $&$\mathrm{iff}$&$   A^\mathbb{B} \leq_k B^\mathbb{B}$&$\mathrm{for} \,\, \mathrm{any}$&$\mathbb{B} \in \mathbb{B}$\\
%$A \fCenter B$&$ \mathrm{and}$&$\neg B \fCenter \neg A $&$\mathrm{iff}$&$   A^\mathbb{B} \leq_t B^\mathbb{B}$&$\mathrm{in}$&$\mathbb{B}$\\
%$A \dashv\vdash B$&$ \mathrm{and}$&$ \neg A\dashv\vdash \neg B $&$\mathrm{iff}$&$  A^\mathbb{B} = B^\mathbb{B}$&$\mathrm{in}$&$\mathbb{B}.$
%\end{tabular}
%\end{center}

\section{Multi-type algebraic presentation}
\label{sec:alg}

In the present section we introduce the algebraic environment which justifies semantically the multi-type approach to bilattice logic presented in Section~\ref{sec:disp}. %In the present section, 
The main insight is that (bounded) bilattices (with conflation) can be equivalently presented as heterogeneous  structures, i.e.~tuples consisting of two (bounded) distributive lattices (De Morgan algebras) together with two maps between them.
%To achieve this, we are going to define the kernel of a bilattice
% and show that it can be endowed with the structure of a bilattice. Then we define two maps between the kernel and the De Morgan lattice reduct of any bilattice: these are the main components of a heterogeneous bilattice. We then show that bilattices can be equivalently presented in terms of heterogeneous bilattices. 
%This will allow us to exploit results from the theory of canonical extensions to obtain
%further insight on the logic of bilattices.
%apply results from the theory of canonical extensions to bilattices.
% Finally, we apply results pertaining to the theory of canonical extensions to heterogenous bilattices. 

\subsection*{Multi-type semantic environment}
For a bilattice $\mathbb{B}$, let $\mathrm{Reg}(\mathbb{B}) = \{ a \in B: a = \neg a\}$ be the set of
\emph{regular elements}~\cite{bou2011varieties}. 
It is easy to show that $\mathrm{Reg}(\mathbb{B})$ is closed under $\band$ and $\bor$, hence $(\mathrm{Reg}(\mathbb{B}), \band, \bor)$ is a sublattice of $(B, \band, \bor)$. For every $a \in B$, we let
$$\mathrm{reg}(a) := (a \vee (a \band \neg a)) \bor \neg(a \vee (a \band \neg a))$$
be the regular element associated with $a$. %(cf.~\cite{bou2011varieties}).
It follows from the representation result of \cite[Theorem 3.2]{bou2011varieties}  that
$$\mathbb{B} \cong (\mathrm{Reg}(\mathbb{B}), \band, \bor) \odot  (\mathrm{Reg}(\mathbb{B}), \band, \bor)$$
where the isomorphism $\pi: B \rightarrow \mathrm{Reg}(\mathbb{B}) \times \mathrm{Reg}(\mathbb{B})$ is defined, for all $a \in B$, as $\pi(a) := \langle\mathrm{reg}(a)$, $ \mathrm{reg}(\neg a)\rangle$. The inverse map $f: \mathrm{Reg}(\mathbb{B}) \times \mathrm{Reg}(\mathbb{B}) \rightarrow B$ is defined, for all $\langle a, b \rangle \in  \mathrm{Reg}(\mathbb{B}) \times \mathrm{Reg}(\mathbb{B})$, as
$$f(\langle a, b \rangle) := (a \band (a \vee b)) \bor (b \band (a \wedge b)). $$

\subsection*{Heterogeneous Bilattices}
\label{Heterogeneous}

\begin{definition}

A distributive lattice $\mathbb{A}$ is  {\em perfect} $(\label{def: perfect algebra}$cf.~$\cite{gehrke2001bounded})$ if
%$\mathbb{A}$
it is complete, completely distributive and completely join-generated by the set $J^{\infty}(\mathbb{A})$ of its completely join-irreducible elements (as well as completely meet-generated by the set  $M^{\infty}(\mathbb{A})$ of its completely meet-irreducible elements).

A lattice isomomorphism $h: \mathbb{L} \rightarrow  \mathbb{L'}$ is  {\em complete} if it satisfies the following properties for each $X\subseteq \mathbb{L}$:
\begin{center}
$h(\bigvee X) = \bigvee h(X) \quad\quad h(\bigwedge X) = \bigwedge h(X)$,
\end{center}
\end{definition}

\begin{definition}
\label{def:hb}
A {\em heterogeneous bilattice} (HBL) is a tuple $\mathbb{H} = (\mathbb{L}_1,~\mathbb{L}_2,~\ud,~\du)$
 satisfying the following conditions: 
\begin{itemize}
\item[($\mathrm{H}1$)] $\mathbb{L}_1$, $\mathbb{L}_2$  are bounded distributive lattices. 
%\item[$\mathrm{H}2$] $\mathbb{L}_1 \cong \mathbb{L}_2$, such that $\ud: \mathbb{L}_1  \to \mathbb{L}_2$ and $\du: \mathbb{L}_2  \to \mathbb{L}_1$ respectively. 
\item[($\mathrm{H}2$)]  $\ud: \mathbb{L}_1  \to \mathbb{L}_2$ and $\du: \mathbb{L}_2  \to \mathbb{L}_1$ are mutually inverse lattice isomorphisms. 
\end{itemize}

An HBL is {\em perfect} if:
\begin{itemize}
\item[($\mathrm{H}3$)]  both $\mathbb{L}_1$ and $\mathbb{L}_2$ are perfect lattices;
\item[($\mathrm{H}4$)]  $\du, \ud$ are complete lattice isomorphisms.
\end{itemize}
\end{definition}

By $(\mathrm{H}2)$ we have %, it immediately follows 
that np = $\mathrm{Id}_{\mathbb{L}_1}$ and pn = $\mathrm{Id}_{\mathbb{L}_2}$. The definition of {\em the heterogeneous bilattice with conflation} (HCBL) is analogous, except that we replace $(\mathrm{H}1)$ with $(\mathrm{H}1')$: $\mathbb{L}_1$ and $\mathbb{L}_2$ are De Morgan algebras.

The following lemma is an easy consequence of the results in~\cite[Theorems 2.3 and 3.2]{gehrke1994bounded}.
%Bounded distributive lattices with operators M Gehrke - Math. Japonica, 1994 - ci.nii.ac.jp

\begin{lemma}
\label{prop:canonical extensions}
If $(\mathbb{L}_1, \mathbb{L}_2, \ud, \du)$ is an HBL %-algebra 
(HCBL%-algebra
),
then $(\mathbb{L}^\delta, \mathbb{D}^\delta, \ud^\delta, \du^\delta)$ is a perfect HBL %-algebra 
(HCBL%-algebra
).
\end{lemma}
$$
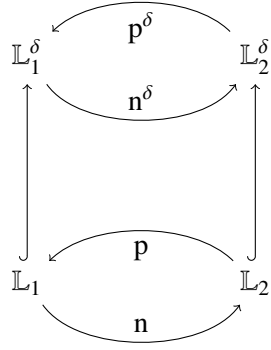
\begin{tikzpicture}[node/.style={circle, draw, fill=black}, scale=1]
\node (DL) at (-1.5,-1.5) {$\mathbb{L}_1$};
\node (DL delta) at (-1.5,1.5) {$\mathbb{L}_1^{\delta}$};
\node (DM) at (1.5,-1.5) {$\mathbb{L}_2$};
\node (DM delta) at (1.5,1.5) {$\mathbb{L}_2^{\delta}$};
\draw [right hook->] (DL) to  (DL delta);
\draw [right hook->] (DM)  to (DM delta);
\draw [->] (DL)  to  [out=-55,in=-125, looseness=0.8] node[above] {$\mathrm{n}$}  (DM);
\draw [->] (DL delta)  to [out=-55,in=-125, looseness=0.8] node[above] {$\mathrm{n}^{\delta}$}  (DM delta);
\draw [->] (DM delta) to [out=135,in=45, looseness=0.8]   node[below] {$\mathrm{p}^{\delta}$}(DL delta);
\draw [->] (DM) to [out=135,in=45, looseness= 0.8]   node[below] {$\mathrm{p}$}  (DL);
\end{tikzpicture}
$$

\subsection*{Equivalence of the two presentations}
\label{Heterogeneous presentation}

The following result is an immediate consequence of Definition~\ref{def:hb}.

\begin{proposition}
\label{prop:from single to multi}
For any bounded distributive bilattice $\mathbb{B}$, 
the tuple
$\mathbb{B}^+ = (\mathbb{L}_1 = \mathrm{Reg}\mathbb{(B)},~\mathbb{L}_2 = \mathrm{Reg}\mathbb{(B)},~\du =  \mathrm{Id}_{\mathrm{Reg}\mathbb{(B)}},~\ud =  \mathrm{Id}_{\mathrm{Reg}\mathbb{(B)}})$ is an HBL., where $\uand = \dand = \band, \uor = \dor = \bor, \utop = \dtop = \top$ and $\ubot = \dbot =\bot.$
%let $\mathbb{L}_1 = \mathbb{L}_2 = (\mathrm{Reg}(\mathbb{B}), \band, \bor, \bkt, \bkb)$ and $p = n = \mathrm{Id}_{\mathrm{Reg}\mathbb{(B)}}$. Then 
%\begin{proof}
%It immediately follows from the result in Section 3.1.
%\end{proof}

\noindent For any CB $\mathbb{B}$, 
$\mathbb{B}^+ = ( \mathbb{L}_1 = ( \mathrm{Reg}\mathbb{(B)}, - ),~\mathbb{L}_2 = ( \mathrm{Reg}\mathbb{(B)}, - ),~\du =  \mathrm{Id}_{\mathrm{Reg}\mathbb{(B)}},~\ud =  \mathrm{Id}_{\mathrm{Reg}\mathbb{(B)}})$ is an HCBL., where $\dneg_2 = \dneg_1 = -$.
%As to distributive bilattices with conflation $\mathbb{B}$, letting $\mathbb{L}_1 = \mathbb{L}_2 = (\mathrm{Reg}(\mathbb{B}), \band, \bor, \bkt, \bkb, -)$ and $p = n = \mathrm{Id}_{\mathrm{Reg}\mathbb{(B)}}$, we have that the tuple $(\mathbb{L}_1,~\mathbb{L}_2,~\ud,~\du)$ 

\end{proposition}

\begin{proposition}
\label{prop:reverse engineering}
If $(\mathbb{L}_1,~\mathbb{L}_2,~n,~p)$ is an HBL (resp.~HCBL),
then $L_1 \times L_2$ can be endowed with the following structure:
\begin{center}
\begin{tabular}{rcl}
$\langle a_1, a_2 \rangle \band \langle b_1, b_2 \rangle$& $:=$&$ \langle a_1 \uand b_1,~a_2 \dand b_2 \rangle$\\

$\langle a_1, a_2 \rangle \bor \langle b_1, b_2 \rangle$& $:=$&$ \langle a_1 \uor b_1,~a_2 \dor b_2 \rangle$\\

$\langle a_1, a_2 \rangle \wedge \langle b_1, b_2 \rangle$& $:=$&$ \langle a_1 \uand b_1,~a_2 \dor b_2 \rangle$\\

$\langle a_1, a_2 \rangle \vee \langle b_1, b_2 \rangle$& $:=$&$ \langle a_1 \uor b_1,~a_2 \dand b_2 \rangle$\\

$\neg \langle a_1, a_2 \rangle $&$ :=$ &$\langle p (a_2), n (a_1) \rangle$\\

$- \langle a_1, a_2 \rangle $&$ :=$ &$\langle \du (\dneg_2a_2), \ud (\dneg_1a_1) \rangle$\\

$\bfb$&$:=$ & $\langle 0,  1\rangle$\\
$\bt$&$:=$ & $\langle 1,  0 \rangle$\\
$\bot$&$:=$ & $\langle 0,  0\rangle$\\
$\top$&$:=$ & $\langle1,  1\rangle$\\

\end{tabular}
\end{center}
%According to the definition of $\mathsf{Proset}$, we can reduce them to
%\begin{itemize}
%\item $(L_1 \times L_2, \leq_t)$ is a partial order, where $\leq_t$ is a truth order defined by $a \leq_t b$ iff $\ubu\pi_1(a) \leq  \ubu\pi_1(b)$ and $\ubu p\pi_2(b) \leq \ubu p\pi_2(a)$; 

%\item $(L_1 \times L_2, \leq_k)$ is a partial order, where $\leq_k$ is a truth order defined by $a \leq_k b$ iff $\ubu\pi_1(a) \leq \ubu\pi_1(b)$ and $\ubu p\pi_2(a) \leq \ubu p\pi_2(b)$; 

%\item If $a \leq_t b$, then $n\pi_1(a) \leq_2 n\pi_1(b)$ and $p\pi_2(b) \leq_1 p\pi_2(a)$; 

%\end{itemize}

%\begin{center}
%\begin{tabular}{|c|c|c|}
%\hline
%& $\mathbb{+}$& $\mathbb{-}$\\
%\hline
%$a \band b$ & $\ubu(\pi_1a \uand \pi_1b)$& $\ubu p(\pi_2a \dand \pi_2b)$\\
%\hline
%$a \bor b$ & $ \ubu(\pi_1a \uor \pi_1b)$& $\ubu p(\pi_2a \dor \pi_2b)$\\
%\hline
%$a \wedge b$& $ \ubu(\pi_1a \uand \pi_1b)$& $\ubu p(\pi_2a \dor \pi_2b)$\\
%\hline
%$a \vee b$ & $\ubu(\pi_1a \uor \pi_1b)$& $\ubu p(\pi_2a \dand \pi_2b)$\\
%\hline
%$a$ & $\ubu\pi_1 a$&$\ubu p \pi_2 a$ \\
%\hline
%\end{tabular}
%\end{center}

%$a \band b := \mathrm{inf_k} \{a, b\} = \ubu(\pi_1a \uand \pi_1b) \,\,\& \,\,\ubu p(\pi_2a \dand \pi_2b)$.

%$a \bor b := \mathrm{sup_k} \{a, b\} = \ubu(\pi_1a \uor \pi_1b) \,\,\& \,\,\ubu p(\pi_2a \dor \pi_2b)$.

%$a \wedge b := \mathrm{inf_t} \{a, b\} = \ubu(\pi_1a \uand \pi_1b) \,\,\& \,\,\ubu p(\pi_2a \dor \pi_2b)$.

%$a \vee b := \mathrm{sup_t} \{a, b\} = \ubu(\pi_1a \uor \pi_1b) \,\,\& \,\,\ubu p(\pi_2a \dand \pi_2b)$.

%$\neg a := \ubu p \pi_2 a$.

\end{proposition}

\begin{proof}
Firstly, we show that $\langle L_1 \times L_2, \band, \bor \rangle$ and $\langle L_1 \times L_2, \wedge, \vee \rangle$ are bounded distributive lattices. It is obvious that they are both bounded lattices. 
We only need to show that the distributivity law holds. We have:
\vspace{-2.5ex}
\begin{center}
\begin{tabular}{llr}
  & $\langle a_1,~a_2 \rangle \band (\langle b_1,~b_2 \rangle \bor (\langle c_1,~c_2 \rangle)$ =&  \\

=\  & $\langle a_1,~a_2 \rangle \band (\langle b_1 \uor c_1 ,~b_2 \dor c_2 \rangle)$ & (Def.~of $\bor$)
 \\

=\  & $\langle a_1 \uand (b_1 \uor c_1),~a_2 \dand (b_2 \dor c_2) \rangle$ & (Def.~of $\band$) \\

=\  & \mc{2}{l}{$\langle (a_1 \uand b_1) \uor (a_1 \uand c_1),~(a_2 \dand b_2) \dor (a_2 \dand c_2) \rangle$} \\
   & \mc{2}{r}{(Distributivity of $\mathbb{L}_1$ and $\mathbb{L}_2$)} \\

=\  & \mc{2}{l}{$\langle (a_1 \uand b_1),~(a_2 \dand b_2) \rangle \bor \langle(a_1 \uand c_1),~(a_2 \dand c_2)\rangle$} \\
   & \mc{2}{r}{(Def.~of $\bor$)} \\

=\  & \mc{2}{l}{$(\langle a_1,~a_2\rangle \band \langle b_1,~b_2\rangle) \bor (\langle a_1,~a_2\rangle \band \langle c_1,~c_2\rangle)$} \\
   & \mc{2}{r}{(Def.~of $\band$)} \\
\end{tabular}
\end{center}
As to $\langle L_1 \times L_2, \wedge, \vee \rangle$, the argument is analogous.

Now we show that the properties of $\neg$ are also met. Assume that $\langle a_1,~a_2\rangle \leq_t \langle b_1,~b_2\rangle$, equivalently, $a_1 \leq_1 b_1$ and $b_2 \leq_2 a_2$. By the definition of $\neg$,  we have $\neg\langle a_1,~a_2\rangle = \langle \du (a_2),~\ud A_1\rangle$ and $\neg\langle b_1,~b_2\rangle = \langle \du (b_2),~\ud (b_1)\rangle$. Hence $\du (b_2) \leq_1 \du (a_2)$ and $\ud A_1 \leq_2 \ud (b_1)$ by (H2). Thus $\neg\langle b_1,~b_2\rangle \leq_t \neg\langle a_1,~a_2\rangle$. 
A similar reasoning shows that the corresponding property involving $\neg$ and $\leq_k$ also holds.
%As to $\neg$ on $\leq_k$, the argument is analogous. 
The following argument shows that $\neg$ is involutive.

 \begin{center}
\begin{tabular}{lll}
&$\neg\neg \langle a_1,~a_2\rangle $ = &\\
=&$ \neg \langle \du (a_2),~\ud A_1\rangle$&Def. of $\neg$\\
=&$ \langle \du\ud A_1,~\ud\du (a_2)\rangle$&Def. of $\neg$\\
=&$ \langle a_1,~a_2\rangle$&np = $\mathrm{Id}_{\mathbb{L}_1}$ and pn = $\mathrm{Id}_{\mathbb{L}_2}$
\end{tabular}
\end{center}

As to conflation, assume $\langle a_1,~a_2\rangle \leq_t \langle b_1,~b_2\rangle$, equivalently, $a_1 \leq_1 b_1$ and $b_2 \leq_2 a_2$. By the definition of $-$ we have  $-\langle a_1,~a_2\rangle = \langle \du (\dneg_2 a_2),~\ud (\dneg_1a_1)\rangle$ and $-\langle b_1,~b_2\rangle = \langle \du (\dneg_2 b_2),~\ud (\dneg_1b_1 )\rangle$. Hence 
$\du (\dneg_2a_2) \leq_1 \du (\dneg_2 b_2)$ and $\ud (\dneg b_1) \leq_2 \ud (\dneg b_2)$ by (H2). Thus $- \langle a_1,~a_2\rangle \leq_t -\langle b_1,~b_2\rangle$. 
A similar reasoning shows that the corresponding property involving $-$ and $\leq_k$ also holds.
%As to conflation on $\leq_k$, the argument is analogous. 
The following arguments show that $-$ is involutive and $-\neg$ are commutative.

 \begin{center}
\begin{tabular}{l@{}ll}
&$--\langle a_1,~a_2\rangle$ = &\\
=\ \ &$-\langle \du (\dneg_2a_2),~\ud (\dneg_1 a_1) \rangle$& (Def.~of $-$) \\
=&$ \langle \du (\dneg_2\ud (\dneg_1a_1)),~\ud (\dneg_1\du (\dneg_2a_2))\rangle$ & (Def.~of $-$) \\
=&$ \langle \du (\dneg_2\dneg_2\ud(a_1)),~\ud (\dneg_1\dneg_1 \du (a_2))\rangle$ & (H2) \\
=&$ \langle \du\ud(a_1),~\ud\du (a_2)\rangle$ & (H1) \\
=&$ \langle a_1,~a_2\rangle$ & (np = $\mathrm{Id}_{\mathbb{L}_1}$, %pn = $\mathrm{Id}_{\mathbb{L}_2}$) 
\\
%= 
& %$ \langle a_1,~a_2\rangle$ 
& %np = $\mathrm{Id}_{\mathbb{L}_1}$, 
pn = $\mathrm{Id}_{\mathbb{L}_2}$). \\

&\\

&$-\neg\langle a_1,~a_2\rangle$ = &\\
=\ \ &$-\langle \du (a_2),~\ud (a_1)\rangle$& (Def.~of $\neg$) \\
=&$ \langle \du (\dneg_2\ud (a_1)),~\ud (\dneg_1\du (a_2))\rangle$ & (Def.~of $-$) \\
=&$ \neg \langle \dneg_1\du (a_2),~\dneg_2\ud (a_1)\rangle$ & (Def.~of $\neg$) \\
=&$ \neg \langle \du (\dneg_2 a_2),~\ud (\dneg_1a_2 )\rangle$ & (H2) \\
=&$\neg-\langle a_1,~a_2\rangle$ & (Def.~of $-$). \\
\end{tabular}
\end{center}
\end{proof}

\begin{definition}
\label{def: aplus}
%For any bounded distributive bilattice $\mathbb{B}$, we let $\mathbb{B}^+ = (\mathrm{Reg}\mathbb{(B)},~\mathrm{Reg}\mathbb{(B)},~\du,~\ud)$, where $\du = \ud = \mathrm{Id}_{\mathrm{Reg}\mathbb{(B)}}$. 
For any HBL $\mathbb{H} = (\mathbb{L}_1,~\mathbb{L}_2,~n,~p)$,  we denote by 
$\mathbb{H}_+ = (B,~\wedge,~\vee,~\band,~\bor,~\neg)$ the product algebra where the four lattice operations are 
defined as in $\mathbb{L}_1 \odot \mathbb{L}_2$ (Theorem~\ref{th:rep}) and the negation is given by 
$\neg \langle a_1, a_2 \rangle := \langle \du (a_2),~\ud A_1 \rangle$ for all $ \langle a_1, a_2 \rangle \in B$.
If $\mathbb{L}_1$ and $\mathbb{L}_2$ are isomorphic De Morgan algebras, then we define 
$\mathbb{H}_+ = (B,~\wedge,~\vee,~\band,~\bor,~\neg,~-)$ as before, with the conflation given by
$- \langle a_1, a_2 \rangle := \langle \du (\dneg_2 a_2),~\ud (\dneg_1a_1)\rangle$ for all $ \langle a_1, a_2 \rangle \in B$.
%
%$\mathbb{H}_+ = (B,~\wedge,~\vee,~\band,~\bor,~\neg)$ is defined as: 
%\begin{itemize}
%\item[$(1)$]  $B = L_1 \times L_2$;
%%$B = \{\langle a_1,~a_2\rangle: a_1 \in L_1,  a_2 \in L_2\}$;
%\item[$(2)$] $\neg: B \rightarrow B$ is defined by  $\neg \langle a_1, a_2 \rangle := \langle pa_2,~na_1\rangle$ for all $ \langle a_1, a_2 \rangle \in B$;
%\item[$(3)$] $\wedge: B \times B \rightarrow B$ is defined by $\langle a_1, a_2 \rangle \wedge \langle b_1, b_2 \rangle := \langle a_1 \uand b_1,~a_2 \dor b_2 \rangle$ for all $ \langle a_1, a_2 \rangle,  \langle b_1, b_2 \rangle \in B$;
%\item[$(4)$] $\vee:  B \times B \rightarrow B$ is defined by $\langle a_1, a_2 \rangle \vee \langle b_1, b_2 \rangle := \langle a_1 \uor b_1,~a_2 \dand b_2 \rangle$ for all $ \langle a_1, a_2 \rangle,  \langle b_1, b_2 \rangle \in B$;
%\item[$(5)$] $\band:  B \times B \rightarrow B$ is defined by $\langle a_1, a_2 \rangle \band \langle b_1, b_2 \rangle := \langle a_1 \uand b_1,~a_2 \dand b_2 \rangle$ for all $ \langle a_1, a_2 \rangle,  \langle b_1, b_2 \rangle \in B$;
%\item[$(6)$] $\bor:  B \times B \rightarrow B$ is defined by $\langle a_1, a_2 \rangle \bor \langle b_1, b_2 \rangle := \langle a_1 \uor b_1,~a_2 \dor b_2 \rangle$ for all $ \langle a_1, a_2 \rangle,  \langle b_1, b_2 \rangle \in B$.
%\end{itemize}
\end{definition}

\begin{proposition}
For any $\mathbb{B} \in \mathsf{B}$ (resp.~$\mathbb{B} \in \mathsf{CB}$) and  and any HBL (resp.~HCBL) $\mathbb{H}$, we have
\[\mathbb{B} \cong (\mathbb{B}^+)_+ \quad \mbox{and}\quad \mathbb{H} \cong (\mathbb{H}_+)^+.\]
\end{proposition}

\begin{proof}
Immediately follows from Propositions~\ref{prop:from single to multi} and~\ref{prop:reverse engineering}.
\end{proof}

%\begin{proposition}
%For all formulas $A, B \in \mathcal{L}$ and every Bilattice $\mathbb{B}$, 
%\[\mathbb{B} \models A \leq_t B  \quad \mathrm{iff} \quad  \mathbb{B^+} \models t_1(A) \leq_1 t_1(B) \quad  \mbox{and} \quad \mathbb{B^+} \models t_1(\neg B) \leq_2 t_1(\neg A)\]
%\[\mathbb{B} \models A \leq_k B  \quad \mathrm{iff} \quad  \mathbb{B^+} \models t_1(A) \leq_1 t_1(B) \quad  \mbox{and} \quad \mathbb{B^+} \models t_1(\neg A) \leq_2 t_1(\neg B)\]

%\end{proposition}

\section{Multi-type proper display calculus}
\label{sec:disp}

In this section we introduce the proper display calculus D.BL (D.CBL) for bilattice logic (with conflation). 
\vspace{-3ex}

\subsection*{Language}

The language $\mathcal{L_{MT}}$ of D.BL is given by the union of the  sets $\mathcal{L}_1$ and $\mathcal{L}_2$ defined as follows. $\mathcal{L}_1$ is given by simultaneous induction over the set $\mathsf{AtProp}_1 = \{p_1, q_1, r_1,\ldots\}$ of $\mathsf{L_1}$-type atomic propositions as follows:
%\begin{itemize}
%\item Structural and operational language of D.BL:
%\end{itemize}
\begin{center}
\begin{tabular}{@{}r@{}c@{}l}
%        & \mc{2}{@{}l}{$\ %\mathsf{Lattice\ \mathsf{L}_1}$} \\
$A_1$ & $\ ::=\ $ & $p_1 \mid \utop \mid \ubot \mid  \du A_2 \mid  A_1 \uand A_1 \mid A_1 \uor A_1$ \\
$X_1$ & $\ ::=\ $ & $A_1 \mid \UTOP \mid \UBOT\mid \DU X_2  \mid X_1 \UAND X_1 \mid X_1 \UOR X_1$ \\% \mid (X_1 \URARR X_1) \mid (X_1 \UCRARR X_1)$ \\
% \\
%        & \mc{2}{@{}l}{$\ \mathsf{Lattice\ \mathsf{L}_2}$} \\
%$A_2$ & $\ ::=\ $ & $p_2 \mid \dtop \mid \dbot \mid \ud A_1  \mid A_2 \dand A_2 \mid A_2 \dor A_2$ \\
%$X_2$ & $\ ::=\ $ & $A_2 \mid \DTOP \mid \DBOT \mid \UD X_1 \mid  X_1 \DAND X_1 \mid X_1 \DOR X_1$ \\ %\mid (X_1 \DRARR X_1) \mid (X_1 \DCRARR X_1)$ \\
\end{tabular}
\end{center}
$\mathcal{L}_2$ is given by simultaneous induction over the set $\mathsf{AtProp}_2 = \{p_2, q_2, r_2,\ldots\}$ of 
$\mathsf{L_2}$-type atomic propositions as follows:
\begin{center}
\begin{tabular}{@{}r@{}c@{}l}
%        & \mc{2}{@{}l}{$\ \mathsf{Lattice\ \mathsf{L}_1}$} \\
%$A_1$ & $\ ::=\ $ & $p_1 \mid \utop \mid \ubot \mid  \du (A_2) \mid  A_1 \uand A_1 \mid A_1 \uor A_1$ \\
%$X_1$ & $\ ::=\ $ & $A_1 \mid \UTOP \mid \UBOT\mid \DU X_2  \mid X_1 \UAND X_1 \mid X_1 \UOR X_1$ \\% \mid (X_1 \URARR X_1) \mid (X_1 \UCRARR X_1)$ \\
% \\
%        & \mc{2}{@{}l}{$\ \mathsf{Lattice\ \mathsf{L}_2}$} \\
$A_2$ & $\ ::=\ $ & $p_2 \mid \dtop \mid \dbot \mid \ud A_1  \mid A_2 \dand A_2 \mid A_2 \dor A_2$ \\
$X_2$ & $\ ::=\ $ & $A_2 \mid \DTOP \mid \DBOT \mid \UD X_1 \mid  X_1 \DAND X_1 \mid X_1 \DOR X_1$ \\ %\mid (X_1 \DRARR X_1) \mid (X_1 \DCRARR X_1)$ \\
\end{tabular}
\end{center}

The language of D.CBL can be obtained by adding structural operators $\ast_1$ and $\ast_2$ and their corresponding connectives $\dneg_1$, $\dneg_2$ to $\mathcal{L}_1$ and $\mathcal{L}_2$ respectively.

\subsection*{Rules}
For $i \in \{1, 2\}$,
\begin{itemize}
\item Pure $\mathsf{L}_i$-type display rules
\end{itemize}
\begin{center}
\begin{tabular}{rl}
\AX $X_i \, \hat{\sqcap}_i \, Y_i \fCenter Z_i $
\LeftLabel{\scriptsize res}
\doubleLine
\UI $X_i \fCenter Y_i \, \check{\sqsupset} \, Z_i $
\DP
 & 
\AX $X_i \fCenter  Y_i \,\check{\sqcup}_i \, Z_i $
\RightLabel{\scriptsize res}
\doubleLine
\UI $X_i \, \hat{\sqsubset}_i \, Y_i  \fCenter  Z_i $
\DP
 \\
\end{tabular}
\end{center}

\begin{itemize}
\item Multi-type display rules
\end{itemize}
\begin{center}
\begin{tabular}{rl}
\AX $\DU X_2 \fCenter Y_1$
\LeftLabel{\scriptsize adj}
\doubleLine
\UI$X_2 \fCenter \UD Y_1$
\DP
 & 
\AX $\UD X_1 \fCenter Y_2$
\RightLabel{\scriptsize adj}
\doubleLine
\UI$X_1 \fCenter \DU Y_2$
\DisplayProof
 \\
\end{tabular}
\end{center}

\begin{itemize}
\item Pure $\mathsf{L}_i$-type identity and cut rules
\end{itemize}
\begin{center}
\begin{tabular}{r@{}l}
\AxiomC{\phantom{$X_i \fCenter A_i$}}
\LeftLabel{\scriptsize $\mathrm{Id_i}$}
\UI$p_i \fCenter p_i$
\DisplayProof
 & 
\AX $X_i \fCenter A_i$
\AX $A_i \fCenter Y_i$
\RightLabel{\scriptsize Cut}
\BI$X_i \fCenter Y_i$
\DisplayProof
 \\
\end{tabular}
\end{center}

\begin{itemize}
\item Pure $\mathsf{L}_i$-type structural rules
\end{itemize}
\begin{center}
\begin{tabular}{rl}

\AX$X_i \,\hat{\sqcap}_i \, \hat{1}_i \fCenter Y_i$
\LeftLabel{\scriptsize $\hat{1}_i$}
\UI$X_i \fCenter Y_i$
\DisplayProof
 & 
\AX$X_i \fCenter Y_i \,\check{\sqcup}_i\, \check{0}_i$
\RightLabel{\scriptsize $\check{0}_i$}
\UI$X_i \fCenter Y_i$
\DisplayProof
\\

 & \\

\AX$X_i \,\hat{\sqcap}_i \,  Y_i \fCenter Z_i$
\LeftLabel{\scriptsize E}
\UI$Y_i \,\hat{\sqcap}_i \,  X_i \fCenter Z_i$
\DisplayProof
 & 
\AX$X_i \fCenter Y_i  \,\check{\sqcup}_i\, Z_i$
\RightLabel{\scriptsize E}
\UI $X_i \fCenter Z_i  \,\check{\sqcup}_i\, Y_i$
\DisplayProof
 \\

 & \\

\AX $(X_i \,\hat{\sqcap}_i \,  Y_i)  \,\hat{\sqcap}_i \,  Z_i \fCenter W_i$
\LeftLabel{\scriptsize A}
\UI$X_i \,\hat{\sqcap}_i \,  (Y_i  \,\hat{\sqcap}_i \,  Z_i) \fCenter W_i$
\DisplayProof
 & 
\AX$X_i \fCenter (Y_i  \,\check{\sqcup}_i\, Z_i)  \,\check{\sqcup}_i\, W_i$
\RightLabel{\scriptsize A}
\UI $X_i \fCenter Y_i  \,\check{\sqcup}_i\, (Z_i  \,\check{\sqcup}_i\,W_i)$
\DisplayProof
 \\

 & \\

\AX$X_i \fCenter Z_i$
\LeftLabel{\scriptsize W}
\UI$X_i \,\hat{\sqcap}_i \,  Y_i \fCenter Z_i$
\DisplayProof
 & 
\AX $X_i \fCenter Y_i$
\RightLabel{\scriptsize W}
\UI $X_i \fCenter Y_i  \,\check{\sqcup}_i\, Z_i$
\DisplayProof
 \\

 & \\

\AX$X_i \,\hat{\sqcap}_i \, X_i \fCenter Z_i$
\LeftLabel{\scriptsize C}
\UI $X_i \fCenter Z_i$
\DisplayProof
 & 
\AX$X_i \fCenter Y_i  \,\check{\sqcup}_i\, Y_i$
\RightLabel{\scriptsize C}
\UI $X_i \fCenter Y_i $
\DisplayProof
 \\
\end{tabular}
\end{center}

\begin{itemize}
\item Pure $\mathsf{L}_i$ type operational rules
\end{itemize}
\begin{center}
\begin{tabular}{r@{}l}

\AX$\hat{1}_i \fCenter X_i$
\LeftLabel{\scriptsize $1_i$}
\UI$1_i \fCenter X_i$
\DisplayProof
\ & \ 
\AxiomC{\phantom{$\BT \fCenter X$}}
\RightLabel{\scriptsize $1_i$}
\UI $\hat{1}_i \fCenter 1_i$
\DisplayProof
 \\

 & \\

\AxiomC{\phantom{$X_i \fCenter \BT$}}
\LeftLabel{\scriptsize $0_i$}
\UI $0_i \fCenter \check{0}_i$
\DisplayProof
 & 
\AX$X_i \fCenter \check{0}_i$
\RightLabel{\scriptsize $0_i$}
\UI$X_i \fCenter 0_i$
\DisplayProof
 \\

&\\ 

\AX$A_i \,\hat{\sqcap}_i \,  B_i \fCenter X_i$
\LeftLabel{\scriptsize$ \sqcap_i$}
\UI$A_i \sqcap_i B_i \fCenter X_i$
\DisplayProof
 & 
\AX$X_i \fCenter A_i$
\AX$Y_i \fCenter B_i$
\RightLabel{\scriptsize $\sqcap_i$}
\BI$X_i \,\hat{\sqcap}_i \,  Y_i \fCenter A_i \sqcap_i B_i$
\DisplayProof
 \\

 & \\

\AX$A_i \fCenter X_i$
\AX$B_i \fCenter Y_i$
\LeftLabel{\scriptsize $\sqcup_i$}
\BI$A_i \sqcup_i B_i\fCenter X_i \,\check{\sqcup}_i\,  Y_i$
\DisplayProof
 & 
\AX$X_i \fCenter A_i \,\check{\sqcup}_i\,  B_i$
\RightLabel{\scriptsize $\sqcup_i$}
\UI$X_i \fCenter A_i \sqcup_i  B_i $
\DisplayProof
 \\
\end{tabular}
\end{center}

\begin{itemize}
\item Multi-type structural rules
\end{itemize}
\begin{center}
\begin{tabular}{rl}
\AX$X_1 \fCenter Y_1$
\LeftLabel{\scriptsize \UD}
\doubleLine
\UI$\UD X_1 \fCenter \UD Y_1$
\DisplayProof
 & 
\AX $X_2 \fCenter Y_2$
\RightLabel{\scriptsize \DU}
\doubleLine
\UI $\DU X_2 \fCenter \DU Y_2$
\DisplayProof
 \\
 & \\
\AX$\UBOT \fCenter X_1$
\LL{\scriptsize $\DU\DBOT$}
\UI$\DU\DBOT \fCenter X_1$
\DP
 & 
\AX$X_1 \fCenter \UTOP$
\RL{\scriptsize $\DU\DTOP$}
\UI$X_1 \fCenter \DU\DTOP$
\DP

 \\
\end{tabular}
\end{center}

\begin{itemize}
\item Multi-type operational rules
\end{itemize}
\begin{center}
\begin{tabular}{rl}

\AX$\UD A_1 \fCenter X_2$
\LeftLabel{\scriptsize \ud}
\UI$\ud A_1 \fCenter X_2$
\DisplayProof
 & 
\AX $X_2 \fCenter \UD A_1$
\RightLabel{\scriptsize  \ud}
\UI $X_2 \fCenter \ud A_1$
\DisplayProof
 \\

 & \\

\AX$\DU A_2 \fCenter X_1$
\LeftLabel{\scriptsize \du}
\UI$\du A_2 \fCenter X_1$
\DisplayProof
 & 
\AX $X_1 \fCenter \DU A_2$
\RightLabel{\scriptsize  \du}
\UI $X_1 \fCenter \du A_2$
\DisplayProof
 \\

\end{tabular}
\end{center}

The multi-type display calculus D.CBL also includes the following rules:

\begin{itemize}
\item Pure $\mathsf{L}_i$ display structural rules:
\end{itemize}
\begin{center}
\begin{tabular}{rl}
\AX $\ast_i X_i \fCenter Y_i$
\LL {\scriptsize adj$\ast$}
\doubleLine
\UI $\ast_i Y_i \fCenter X_i$
\DP
 & 
\AX $X_i \fCenter \ast_iY_i$
\RL {\scriptsize adj$\ast$}
\doubleLine
\UI $Y_i \fCenter \ast_iX_i$
\DP
 \\
\end{tabular}
\end{center}

\begin{itemize}
\item Pure $\mathsf{L}_i$ structural rules:
\end{itemize}
\begin{center}
\begin{tabular}{c}
\AX $X_i \fCenter Y_i$
\LL {\scriptsize cont}
\doubleLine
\UI $\ast_i Y_i \fCenter \ast_i X_i$
\DP
 \\
\end{tabular}
\end{center}

\begin{itemize}
\item Multi-type structural rules:
\end{itemize}
\begin{center}
\begin{tabular}{cc}
\AX$\UD\ast_1 X_1 \fCenter Y_2$
\LL{\scriptsize{$\ast_2 \UD$}}
\UI$\ast_2\UD X_1 \fCenter Y_2$
\DP
 & 
\AX$X_2 \fCenter \UD \ast_1 Y_1$
\RL{\scriptsize{$\ast_2 \UD$}}
\UI$X_2 \fCenter \ast_2\UD Y_1$
\DP
 \\
\end{tabular}
\end{center}

\begin{itemize}
\item Pure $\mathsf{L}_i$ operational rules:
\end{itemize}
\begin{center}
\begin{tabular}{rl}
\AX $\ast_i X_i \fCenter Y_i$
\LL {\scriptsize $\dneg_i$}
\UI $\dneg_i X_i \fCenter Y_i$
\DP
 & 
\AX $X_i \fCenter \ast_iY_i$
\RL {\scriptsize $\dneg_i$}
\UI $X_i \fCenter \dneg_iY_i$
\DP
\end{tabular}
\end{center}

%\begin{remark}
An essential feature of our calculus is that the logical rules are standard introduction rules of display calculi. This is key for achieving a canonical proof of cut elimination. The special behaviour of negation is captured by a suitable translation in a multi-type environment, which makes it possible to circumvent the technical difficulties created by the non-standard introduction rules of \cite{arieli1996reasoning}. 
%\end{remark}

\section{Properties}
\label{sec:prop}

\subsection*{Soundness}

We outline the  verification of soundness of the rules of D.BL (resp.~D.CBL) w.r.t.~the semantics of {\em perfect} HBL%-algebras 
(resp.~
HCBL). The first step consists in interpreting structural symbols as logical symbols according to their (precedent or succedent) position. This makes it possible to interpret sequents as inequalities, and rules as quasi-inequalities. The verification of  soundness of the rules of D.BL (resp.~D.CBL) then consists in checking the validity of their corresponding quasi-inequalities in perfect HBL (resp.~HCBL). For example, the rules on the left-hand side below are interpreted as the quasi-inequalities on the right-hand side:

\begin{center}
\begin{tabular}{r@{}c@{}l}
\AX $\DU X_2 \fCenter Y_1$
\doubleLine
\UI$X_2 \fCenter \UD Y_1$
\DP
&$\ \rightsquigarrow\ \,$&
$\forall a_1 \forall a_2 [\du (a_2) \leq_1 b_1 \Leftrightarrow a_2 \leq_2 \ud (b_1)]$
\\

 &&\\
\AX $X_i \fCenter Y_i$
\doubleLine
\UI $\ast_i Y_i \fCenter \ast_i X_i$
\DP&$\ \rightsquigarrow\ \,$&
$\forall a_i \forall b_i[a_i \leq_i b_i \Leftrightarrow \dneg_i b_i \leq_i \dneg_i a_i]$
\\
\end{tabular}
\end{center}

The verification of soundness of pure-type rules and of the introduction rules following this procedure is routine, and is omitted.
The validity of the quasi-inequalities corresponding to multi-type structural rules follows straightforwardly from the observation that the quasi-inequality corresponding to each rule is obtained by running the algorithm ALBA \cite[Section 3.4]{GMPTZ} on one of the  defining inequalities of HBL %-algebras 
(resp.~HCBL%-algebras
)\footnote{As discussed in \cite{GMPTZ}, the soundness of the rewriting rules of ALBA only depends on the order-theoretic properties of the interpretation of the logical connectives and their adjoints and residuals. The fact that some of these maps are not internal operations but have different domains and codomains does not make any substantial difference.}.  For instance, the soundness of the first rule above is due to $\du$ and $\ud$ being isomorphisms (by (H2) 
in Definition~\ref{def:hb}).

\subsection*{Completeness}\label{ssec: completeness}
In order to prove completeness, we shall
introduce translations %$t_1$ 
from
%We shall translate 
sequents %$A\vdash B$ 
in the language of BL (resp.~CBL) into sequents %$t_1(A) \vdash t_1(B)$
 in the language of D.BL (resp.~D.CBL).
% by means of the following translations. 

Let $t_1(\cdot), t_2(\cdot): \mathcal{L} \rightarrow \mathcal{L_{MT}}$ be maps between the  language $\mathcal{L}$ of BL and $\mathcal{L_{MT}}$ of D.BL  inductively defined as follows:

\begin{center}
\begin{tabular}{rcl|rcl}
$t_1(p)$ &$\, :=\, $ & $p_1$&$t_2(p)$ &$\, :=\, $ & $p_2$\\
$t_1(\bt)$& $:= $&$\utop$&$t_2(\bt)$& $:=$ &$\dbot$\\
 $t_1(\bfb)$& $:=$&$ \ubot$&$t_2(\bfb)$& $:= $&$\dtop$\\
$t_1(\bkt)$& $:=$&$ \utop $&$t_2(\bkt)$&$ :=$&$ \dtop$\\
$t_1(\bkb)$&$ :=$&$ \ubot $&$t_2(\bkb)$& := &$\dbot$\\
$t_1(A \wedge B)$ & $:= $ & $t_1(A) \uand t_1(B)$
&$t_2(A \wedge B)$ & $:= $ & $ t_2(A) \dor t_2(B) $\\
$t_1(A \vee B)$ & $:= $ & $t_1(A) \uor t_1(B)$
&$t_2(A \vee B)$ & $:= $ & $t_2(A) \dand t_2(B)$\\
$t_1(A \band B)$ & $:= $ & $t_1(A) \uand  t_1(B)$
&$t_2(A \band B)$ & $:= $ & $t_2(A) \dand t_2(B) $\\
$t_1(A \bor B)$ & $:= $ & $t_1(A) \uor t_1(B)$
&$t_2(A \bor B)$ & $:= $ & $t_2(A) \dor t_2(B)$\\
$t_1(\neg A)$& $:=$ &$\du t_2(A)$&$t_2(\neg A)$ &:= &$\ud t_1(A)$
\end{tabular}
\end{center}
A sequent
$A\vdash B$ of BL is translated as
$t_1(A) \vdash t_1(B)$  of D.BL.
For CBL we also need the following translation for the conflation connective:
\begin{center}
\begin{tabular}{rcl|rcl}
$\ t_1(- A)$ & $\, :=\, $ & $\du\dneg_2 t_2(A)$ & $t_2(- A)$ & $\, :=\, $ & $\ud \dneg_1 t_1(A)$ \\
\end{tabular}
\end{center}

The following proposition is immediate. 

\begin{proposition}
\label{prop1}
For every formula $A$ of $\mathrm{BL}$ (resp.~CBL), the sequents $t_1(A) \fCenter t_1(A) $ and $t_2(A) \fCenter t_2(A)$ are derivable in D.BL (resp.~D.CBL).
\end{proposition}

\begin{proof}
By induction on the complexity of the formula $A$. If $A$ is an atomic formula, the translation of $t_i(A) \fCenter t_i(A)$ with $i\in\{1,2\}$ is $A_i \fCenter A_i$, hence it is derivable using (Id) in $\mathsf{L_1}$ and $\mathsf{L_2}$, respectively. If $A = A_1 \band A_2$, then $t_i(A_1 \band A_2) = t_i(A_1) \uand t_i(A_2)$ and if $A=A_1 \bor A_2$, then $t_i(A_1 \bor A_2) = t_i(A_1) \uor t_i(A_2)$. By induction hypothesis, $t_i(A_i) \fCenter t_i(A_i) $. So, it is enough to show that:
{
\[
\AX $t_i(A_1) \fCenter t_i(A_1) $
\UI $t_i(A_1) \hat{\sqcap}_i t_i(A_2) \fCenter t_i(A_1) $
\AX $t_i(A_2) \fCenter t_i(A_2) $
\UI $t_i(A_2) \hat{\sqcap}_i t_i(A_1) \fCenter t_i(A_2) $
\LL{\scriptsize $\mathrm{E}$}
\UI $t_i(A_1) \hat{\sqcap}_i t_i(A_2) \fCenter t_i(A_2) $
\BI $(t_i(A_1) \hat{\sqcap}_i t_i(A_2)) \hat{\sqcap}_i (t_i(A_1) \hat{\sqcap}_i t_i(A_2))  \fCenter t_i(A_1) \sqcap_i t_i(A_2)$
\LL{\scriptsize $\mathrm{W}$}
\UI $t_i(A_1)  \hat{\sqcap}_i t_i(A_2)  \fCenter t_i(A_1) \sqcap_i t_i(A_2)$
\UI $t_i(A_1) \sqcap_i t_i(A_2)  \fCenter t_i(A_1) \sqcap_i t_i(A_2)$
\DP
\]

\[
\AX $t_i(A_1) \fCenter t_i(A_1) $
\UI $t_i(A_1) \fCenter t_i(A_1)  \UOR t_i(A_2)$
\AX $t_i(A_2) \fCenter t_i(A_2) $
\UI $t_i(A_2)  \fCenter t_i(A_2) \UOR t_i(A_1)$
\RL{\scriptsize $\mathrm{E}$}
\UI $t_i(A_1) \fCenter t_i(A_1)  \UOR t_i(A_2) $
\BI $t_i(A_1) \uor t_i(A_2) \fCenter (t_i(A_1)  \UOR t_i(A_2)) \UOR (t_i(A_1)  \UOR t_i(A_2))$
\RL{\scriptsize $\mathrm{C}$}
\UI $t_i(A_1) \uor t_i(A_2) \fCenter t_i(A_1)  \UOR t_i(A_2)$
\UI $t_i(A_1) \uor t_i(A_2)  \fCenter t_i(A_1) \uor t_i(A_2)$
\DP
\]
}

The arguments for $A = A_1 \wedge A_2$ and $A = A_1 \vee A_2$ are similar and they are omitted.

If $A = \neg B$, then $t_1(\neg B) = \du t_2(B)$ and $t_2(\neg B) = \ud t_1(B)$. By induction hypothesis $t_i(A) \fCenter t_i(A)$. Hence it is enough to show that:
\[
\begin{tabular}{cc}
\AX $t_2(B) \fCenter t_2(B) $
\RL{\scriptsize $\DU$}
\UI$\DU t_2(B) \fCenter \DU t_2(B) $
\UI$\DU t_2(B) \fCenter \du t_2(B) $
\UI$\du t_2(B) \fCenter \du t_2(B) $
\DP
 & 
\AX $t_1(B) \fCenter t_1(B) $
\LL{\scriptsize $\UD$}
\UI$\UD t_1(B) \fCenter \UD t_1(B) $
\UI$\UD t_1(B) \fCenter \ud t_1(B) $
\UI$\ud t_1(B) \fCenter \ud t_1(B) $
\DP
\end{tabular}
\]

If $A = -B$, then $t_1(- B) = \du\sim_2t_2(B)$ and $t_2(- B) = \ud\sim_1 t_1(B)$. By induction hypothesis $t_i(B) \fCenter t_i(B)$. Hence it is enough to show that:
\[
\begin{tabular}{cc}
\AX $t_2(B) \fCenter t_2(B) $
\RL{\scriptsize cont}
\UI $\ast_2 t_2(B) \fCenter \ast_2 t_2(B) $
\UI $\ast_2 t_2(B) \fCenter \sim_2 t_2(B) $
\UI $\sim_2 t_2(B) \fCenter \sim_2 t_2(B) $
\RL{\scriptsize \DU}
\UI $\DU\sim_2 t_2(B) \fCenter \DU\sim_2 t_2(B) $
\UI $\du\sim_2 t_2(B) \fCenter \DU\sim_2 t_2(B) $
\UI $\du\sim_2 t_2(B) \fCenter \du\sim_2 t_2(B) $
\DP
 & 
\AX $t_1(B) \fCenter t_1(B) $
\RL{\scriptsize cont}
\UI $\ast_1 t_1(B) \fCenter \ast_1 t_1(B) $
\UI $*_1 t_1(B) \fCenter \sim_1 t_1(B) $
\UI $\sim_1 t_2(B) \fCenter \sim_1 t_1(B) $
\LL{\scriptsize \UD}
\UI $\UD\sim_1 t_1(B) \fCenter \UD\sim_1 t_1(B) $
\UI $\ud\sim_1 t_1(B) \fCenter \ud\sim_1 t_1(B) $
\UI $\ud\sim_1 t_1(B) \fCenter \ud\sim_1 t_1(B) $
\DP
\end{tabular}
\]

\end{proof}

\begin{proposition}
\label{prop2}
For all formulas $A, B$ of $\mathrm{BL}$ (resp.~CBL), if $A \fCenter B$ is derivable in BL (resp.~CBL), then $t_1(A) \fCenter t_1(B) $  is derivable in D.BL (resp.~D.CBL).
\end{proposition}

\begin{proof}

In what follows we show that the translations of the axioms and rules of BL (resp~C.BL) are derivable in D.BL (resp.~D.CBL). Since BL (resp~C.BL) is complete w.r.t. the class of bilattice algebras (by Theorem \ref{thm:com}), and hence w.r.t their associated heterogeneous algebras (by Propositions \ref{prop:from single to multi} and \ref{prop:reverse engineering}), this is enough to show the completeness of D.BL (resp.~D.CBL). 
For the sake of readability, although each formula $A$ in precedent (resp.~succedent) position should be written as $t_1(A)$, we suppress it in the derivation trees of the axioms. 

The Identity axiom $A \vdash A$ is proved in Proposition \ref{prop1}. 

The derivations of the binary rules are standard and we omit them. 

The translations of the axioms $\bfb \fCenter A$, $A \fCenter \bt$, $\bkb \fCenter A$, $A \fCenter \bkt$ are $0_1 \fCenter A_1$, $A_1 \fCenter 1_1$, $0_1 \fCenter A_1$, and $A_1 \fCenter 1_1$, respectively. The derivations are straightforward and they are omitted, in particular they make use of the introduction rules of $1_1$ and $0_1$, Weakening (W) and the structural rules for the neutral element ($\UTOP$ and $\UBOT$).

The translations of the axioms $A \fCenter \neg\bfb$, $\neg\bt \fCenter A$, $\neg\bkb \fCenter A$, $A \fCenter \neg\bkt$ are $A_1 \fCenter \du 1_2$, $\du 0_2 \fCenter A_1$, $\du 0_2 \fCenter A_1,$ and $A_1 \fCenter \du 1_2$, respectively. The derivations are straightforward and they are omitted, in particular they make use of the introduction rules of $1_2$ and $0_2$, the structural rules $\DU\DBOT$ and $\DU\DTOP$. 

$\neg\neg A \dashv\vdash A \ \  \rightsquigarrow \ \ \du\ud A_1 \fCenter A_1$ and $A_1 \dashv\vdash \du\ud A_1 $

\begin{center}
{\fns 
\begin{tabular}{cc}
\AXC{$ A_1  \fCenter A_1$}
\RL{\scriptsize \UD}
\UIC{$ \UD A_1  \fCenter \UD A_1$}
\UIC{$ \ud A_1  \fCenter \UD A_1$}
\LL{\scriptsize adj}
\UIC{$\DU\ud A_1  \fCenter A_1$}
\UIC{$\du\ud A_1  \fCenter A_1$}
\DP
&
\AXC{$ A_1  \fCenter A_1$}
\LL{\scriptsize \UD}
\UIC{$ \UD A_1  \fCenter \UD A_1$}
\UIC{$ \UD A_1  \fCenter \ud A_1$}
\RL{\scriptsize adj}
\UIC{$A_1  \fCenter \DU\ud A_1$}
\UIC{$A_1  \fCenter \du\ud A_1$}
\DP
\end{tabular}
 }
\end{center}

$-- A \dashv\vdash A \ \ \rightsquigarrow \ \ \du\sim_2\ud\sim_1 A_1 \dashv\vdash A_1$ 

\begin{center}
{\fns 
\begin{tabular}{cc}
\AX$ A_1  \fCenter A_1$
\LL{\scriptsize cont}
\UI$*_1 A_1 \fCenter *_1 A_1$
\UI$\ast_1 A_1 \fCenter \sim_1 A_1$
\LL{\scriptsize \UD}
\UI$\UD*_1 A_1 \fCenter \UD\sim_1 A_1$
\UI$\UD*_1 A_1 \fCenter \ud\sim_1 A_1$
\LL{\scriptsize {$*_2 \UD$}}
\UI$*_2\UD A_1 \fCenter \ud\sim_1 A_1$
\LL{\scriptsize adj$\ast$}
\UI$*_2\ud\sim_1 A_1 \fCenter \UD A_1$
\UI$\sim_2\ud\sim_1 A_1 \fCenter \UD A_1$
\LL{\scriptsize adj}
\UI$\DU\sim_2\ud\sim_1 A_1 \fCenter A_1$
\UI$\du\sim_2\ud\sim_1 A_1 \fCenter A_1$
\DP
&
\AX$ A_1  \fCenter A_1$
\LL{\scriptsize cont}
\UI$*_1 A_1 \fCenter *_1 A_1$
\UI$\sim_1 A_1 \fCenter *_1 A_1$
\LL{\scriptsize \UD}
\UI$\UD\sim_1 A_1 \fCenter \UD*_1 A_1$
\UI$\ud\sim_1 A_1 \fCenter \UD*_1 A_1$
\RL{\scriptsize{$*_2 \UD$}}
\UI$\ud\sim_1 A_1 \fCenter *_2\UD A_1$
\RL{\scriptsize adj$\ast$}
\UI$\UD A_1 \fCenter *_2\ud\sim_1 A_1$
\UI$\UD A_1 \fCenter \sim_2\ud\sim_1 A_1$
\RL{\scriptsize adj}
\UI$A_1 \fCenter \DU\sim_2\ud\sim_1 A_1$
\UI$A_1 \fCenter \du\sim_2\ud\sim_1 A_1$
\DP
\end{tabular}
 }
\end{center}

$-\neg A \dashv\vdash \neg- A \ \ \rightsquigarrow \ \ \du\sim_2\ud A_1 \dashv\vdash \du\ud\sim_1 A_1$

\begin{center}
{\fns 
\begin{tabular}{cc}
\AX$ A_1  \fCenter A_1$
\RL{\scriptsize cont}
\UI$ *_1 A_1 \fCenter *_1 A_1$
\UI$ *_1 A_1 \fCenter \sim_1 A_1$
\LL{\scriptsize *}
\UI$ *_1\sim_1 A_1 \fCenter A_1$
\LL{\scriptsize  \UD}
\UI$ \UD*_1\sim_1 A_1 \fCenter \UD A_1$
\UI$ \UD*_1\sim_1 A_1 \fCenter \ud A_1$
\LL{\scriptsize {$\ast_2 \UD$}}
\UI$*_2 \UD\sim_1 A_1 \fCenter \ud A_1$
\LL{\scriptsize adj$\ast$}
\UI$*_2\ud A_1 \fCenter \UD\sim_1 A_1$
\UI$\sim_2\ud A_1 \fCenter \UD\sim_1 A_1$
\UI$\sim_2\ud A_1 \fCenter \ud\sim_1 A_1$
\RL{\scriptsize \DU}
\UI$\DU\sim_2\ud A_1 \fCenter \DU\ud\sim_1 A_1$
\UI$\DU\sim_2\ud A_1 \fCenter \du\ud\sim_1 A_1$
\UI$\du\sim_2\ud A_1 \fCenter \du\ud\sim_1 A_1$
\DP
&
\AX$ A_1  \fCenter A_1$
\RL{\scriptsize cont}
\UI$*_1 A_1 \fCenter  *_1 A_1$
\UI$\sim_1 A_1 \fCenter  *_1 A_1$
\RL{\scriptsize adj$\ast$}
\UI$A_1 \fCenter  *_1\sim_1 A_1$
\LL{\scriptsize {$\UD$}}
\UI$\UD A_1 \fCenter  \UD*_1\sim_1 A_1$
\UI$\ud A_1 \fCenter  \UD*_1\sim_1 A_1$
\RL{\scriptsize {$\ast_2 \UD$}}
\UI$\ud A_1 \fCenter *_2 \UD\sim_1 A_1$
\RL{\scriptsize adj$\ast$}
\UI$ \UD\sim_1 A_1 \fCenter *_2\ud A_1$
\UI$ \ud\sim_1 A_1 \fCenter *_2\ud A_1$
\UI$ \ud\sim_1 A_1 \fCenter \sim_2\ud A_1$
\RL{\scriptsize \DU}
\UI$ \DU\ud\sim_1 A_1 \fCenter \DU\sim_2\ud A_1$
\UI$ \DU\ud\sim_1 A_1 \fCenter \du\sim_2\ud A_1$
\UI$ \du\ud\sim_1 A_1 \fCenter \du\sim_2\ud A_1$
\DP
\end{tabular}
 }
\end{center}

$ \neg A  \wedge \neg B \dashv\vdash \neg (A \vee B) \ \ \rightsquigarrow \ \ \du A_2 \uand \du B_2 \dashv\vdash \du(A_2 \dand B_2) $ \ \ and 

$\neg A  \band \neg B \dashv\vdash \neg (A \band B) \ \ \rightsquigarrow \ \ \du A_2 \uand \du B_2 \dashv\vdash \du(A_2 \dand B_2) $

\begin{center}
{
\begin{tabular}{c}
\AX $A_2 \fCenter A_2$
\RL{\scriptsize $\DU$}
\UI $\DU A_2 \fCenter \DU A_2$
\UI $\du A_2 \fCenter \DU A_2$
\LL{\scriptsize W}
\UI $\du A_2 \UAND \du B_2 \fCenter \DU A_2$
\UI $\du A_2 \uand \du B_2 \fCenter \DU A_2$
\RL{\scriptsize adj}
\UI $\UD(\du A_2 \uand \du B_2) \fCenter A_2$
\AX $B_2 \fCenter B_2$
\RL{\scriptsize $\DU$}
\UI $\DU B_2 \fCenter \DU B_2$
\UI $\du B_2 \fCenter \DU B_2$
\LL{\scriptsize W}
\UI $\du B_2 \UAND \du A_2 \fCenter \DU B_2$
\LL{\scriptsize E}
\UI $\du A_2 \uand \du B_2 \fCenter \DU B_2$
\RL{\scriptsize adj}
\UI $\UD(\du A_2 \uand \du B_2) \fCenter B_2$
\BI $\UD(\du A_2 \uand \du B_2) \DAND \UD(\du A_2 \uand \du B_2) \fCenter A_2 \dand B_2 $
\LL{\scriptsize C}
\UI $\UD(\du A_2 \uand \du B_2) \fCenter A_2 \dand B_2 $
\RL{\scriptsize adj}
\UI $\du A_2 \uand \du B_2 \fCenter \DU(A_2 \dand B_2) $
\UI $\du A_2 \uand \du B_2 \fCenter \du(A_2 \dand B_2) $
\DP
\end{tabular}
 }
\end{center}

%$ \neg (A \vee B) \dashv\vdash \neg A  \wedge \neg B \ \ \rightsquigarrow \ \ \du(A_2 \dand B_2)  \dashv\vdash \du A_2 \uand \du B_2$

\begin{center}
{
\begin{tabular}{c}
\AX$A_2 \fCenter A_2$
\LL{\scriptsize W}
\UI$A_2 \DAND B_2 \fCenter A_2$
\UI$A_2 \dand B_2 \fCenter A_2$
\RL{\scriptsize $\DU$}
\UI$\DU(A_2 \dand B_2) \fCenter \DU A_2$
\UI$\DU(A_2 \dand B_2) \fCenter \du A_2$
\UI$\du(A_2 \dand B_2) \fCenter \du A_2$
\AX$B_2 \fCenter B_2$
\LL{\scriptsize W}
\UI$B_2 \DAND A_2 \fCenter B_2$
\LL{\scriptsize E}
\UI$A_2 \dand B_2 \fCenter B_2$
\RL{\scriptsize $\DU$}
\UI$\DU(A_2 \dand B_2) \fCenter \DU B_2$
\UI$\DU(A_2 \dand B_2) \fCenter \du B_2$
\UI$\du(A_2 \dand B_2) \fCenter \du B_2$
\BI $\du(A_2 \dand B_2) \UAND \du(A_2 \dand B_2)  \fCenter \du A_2 \uand \du B_2$
\LL{\scriptsize C}
\UI $\du(A_2 \dand B_2)  \fCenter \du A_2 \uand \du B_2$
\DP
\end{tabular}
 }
\end{center}

$ \neg (A  \wedge B) \dashv\vdash \neg A \vee \neg B \ \ \rightsquigarrow \ \ \du (A_2 \dor B_2) \dashv\vdash \du A_2 \uor \du B_2 $\ \ and

$\neg (A \bor B)  \dashv\vdash \neg A  \bor \neg B\ \ \rightsquigarrow \ \ \du (A_2 \dor B_2) \dashv\vdash \du A_2 \uor \du B_2 $

\begin{center}
{
\begin{tabular}{c}
\AX$A_2 \fCenter A_2$
\RL{\scriptsize $\DU$}
\UI$\DU A_2 \fCenter \DU A_2$
\UI$\DU A_2 \fCenter \du A_2$
\RL{\scriptsize W}
\UI$\DU A_2 \fCenter \du A_2 \UOR \du B_2$
\LL{\scriptsize adj}
\UI$A_2 \fCenter \UD(\du A_2 \UOR \du B_2)$

\AX$B_2 \fCenter B_2$
\RL{\scriptsize $\DU$}
\UI$\DU B_2 \fCenter \DU B_2$
\UI$\DU B_2 \fCenter \du B_2$
\RL{\scriptsize W}
\UI$\DU A_2 \fCenter \du B_2 \UOR \du A_2$
\RL{\scriptsize E}
\UI$\DU A_2 \fCenter \du A_2 \UOR \du B_2$
\LL{\scriptsize adj}
\UI$A_2 \fCenter \UD(\du A_2 \UOR \du B_2)$
\BI $A_2 \dor B_2 \fCenter \UD(\du A_2 \UOR \du B_2) \UOR \UD(\du A_2 \UOR \du B_2) $
\RL{\scriptsize C}
\UI $A_2 \dor B_2 \fCenter \UD(\du A_2 \UOR \du B_2) $
\LL{\scriptsize adj}
\UI $\DU (A_2 \dor B_2) \fCenter \du A_2 \UOR \du B_2 $
\UI $\du (A_2 \dor B_2) \fCenter \du A_2 \UOR \du B_2 $
\UI $\du (A_2 \dor B_2) \fCenter \du A_2 \uor \du B_2 $
\DP
\end{tabular}
 }
\end{center}

%$\neg A \vee \neg B  \fCenter \neg (A  \wedge B) \ \ \rightsquigarrow \ \ \du A_2 \uor \du B_2 \fCenter \du (A_2 \dor B_2)$

\begin{center}
{
\begin{tabular}{c}
\AX$A_2 \fCenter A_2$
\RL{\scriptsize $\DU$}
\UI $\DU A_2  \fCenter \DU A_2 $
\UI $\du A_2  \fCenter \DU A_2 $
\RL{\scriptsize adj}
\UI $\UD\du A_2  \fCenter A_2 $
\RL{\scriptsize W}
\UI $\UD\du A_2  \fCenter A_2 \DOR B_2$
\UI $\UD\du A_2  \fCenter A_2 \dor B_2$
\RL{\scriptsize adj}
\UI $\du A_2  \fCenter \DU (A_2 \dor B_2)$
\UI $\du A_2  \fCenter \du (A_2 \dor B_2)$

\AX$B_2 \fCenter B_2$
\RL{\scriptsize $\DU$}
\UI $\DU B_2  \fCenter \DU B_2 $
\UI $\du B_2  \fCenter \DU B_2 $
\RL{\scriptsize adj}
\UI $\UD\du B_2  \fCenter B_2 $
\RL{\scriptsize W}
\UI $\UD\du B_2  \fCenter B_2 \DOR A_2$
\RL{\scriptsize E}
\UI $\UD\du B_2  \fCenter A_2 \dor B_2$
\RL{\scriptsize adj}
\UI $\du B_2  \fCenter \DU (A_2 \dor B_2)$
\UI $\du B_2  \fCenter \du (A_2 \dor B_2)$
\BI $\du A_2 \uor \du B_2 \fCenter \du (A_2 \dor B_2) \UOR \du (A_2 \dor B_2)$
\RL{\scriptsize C}
\UI $\du A_2 \uor \du B_2 \fCenter \du (A_2 \dor B_2)$
\DP
\end{tabular}
 }
\end{center}

$-(A \wedge B)  \dashv\vdash - A  \wedge - B \ \ \rightsquigarrow \ \ \du\sim_2(A_2 \dor B_2) \dashv\vdash \du\sim_2A_2 \uand \du\sim_2B_2$

\begin{center}
{
\begin{tabular}{c}
\AXC{$ A_2  \fCenter A_2$}
\RL{\scriptsize W}
\UI $A_2   \fCenter  A_2 \DOR B_2$
\UI $A_2   \fCenter  A_2 \dor B_2$
\LL{\scriptsize cont}
\UI $ *_2(A_2 \dor B_2)  \fCenter *_2A_2$
\UI $ *_2(A_2 \dor B_2)  \fCenter \sim_2A_2$
\UI $ \sim_2(A_2 \dor B_2)  \fCenter \sim_2A_2$
\RL{\scriptsize \DU}
\UI $ \DU\sim_2(A_2 \dor B_2)  \fCenter \DU\sim_2A_2$
\UI $ \DU\sim_2(A_2 \dor B_2)  \fCenter \du\sim_2A_2$
\UI $ \du\sim_2(A_2 \dor B_2)  \fCenter \du\sim_2A_2$
\AXC{$ B_2  \fCenter B_2$}
\RL{\scriptsize W}
\UI $B_2   \fCenter  B_2 \DOR A_2$
\RL{\scriptsize E}
\UI $B_2   \fCenter  A_2 \DOR B_2$
\UI $B_2   \fCenter  A_2 \dor B_2$
\LL{\scriptsize cont}
\UI $ *_2(A_2 \dor B_2)  \fCenter *_2B_2$
\UI $ *_2(A_2 \dor B_2)  \fCenter \sim_2B_2$
\UI $ \sim_2(A_2 \dor B_2)  \fCenter \sim_2B_2$
\RL{\scriptsize \DU}
\UI $ \DU\sim_2(A_2 \dor B_2)  \fCenter \DU\sim_2B_2$
\UI $ \DU\sim_2(A_2 \dor B_2)  \fCenter \du\sim_2B_2$
\UI $ \du\sim_2(A_2 \dor B_2)  \fCenter \du\sim_2B_2$
\BIC{$ \du\sim_2(A_2 \dor B_2) \UAND  \du\sim_2(A_2 \dor B_2) \fCenter \du\sim_2A_2 \uand \du\sim_2B_2$}
\LL{\scriptsize C}
\UIC{$ \du\sim_2(A_2 \dor B_2) \fCenter \du\sim_2A_2 \uand \du\sim_2B_2$}
\DP
\end{tabular}
 }
\end{center}

%$- A  \wedge - B \fCenter -(A \wedge B) \ \ \rightsquigarrow \ \ \du\sim_2A_2 \uand \du\sim_2B_2 \fCenter \du\sim_2(A_2 \dor B_2)$

\begin{center}
{
\begin{tabular}{c}
\AXC{$ A_2  \fCenter A_2$}
\RL{\scriptsize W}
\UIC{$ A_2  \fCenter A_2 \DOR B_2$}
\UIC{$ A_2  \fCenter A_2 \dor B_2$}
\LL{\scriptsize cont}
\UIC{$ \ast_2(A_2 \dor B_2)  \fCenter \ast_2A_2$}
\UIC{$ \ast_2(A_2 \dor B_2)  \fCenter \sim_2A_2$}
\UIC{$ \sim_2(A_2 \dor B_2)  \fCenter \sim_2A_2$}
\RL{\scriptsize \DU}
\UIC{$ \DU\sim_2(A_2 \dor B_2)  \fCenter \DU\sim_2A_2$}
\UIC{$ \DU\sim_2(A_2 \dor B_2)  \fCenter \du\sim_2A_2$}
\UIC{$ \du\sim_2(A_2 \dor B_2)  \fCenter \du\sim_2A_2$}
\AXC{$ B_2  \fCenter B_2$}
\RL{\scriptsize W}
\UIC{$ B_2  \fCenter B_2 \DOR A_2$}
\RL{\scriptsize E}
\UIC{$ B_2  \fCenter A_2 \DOR B_2$}
\UIC{$ B_2  \fCenter A_2 \dor B_2$}
\LL{\scriptsize cont}
\UIC{$ \ast_2(A_2 \dor B_2)  \fCenter \ast_2B_2$}
\UIC{$ \ast_2(A_2 \dor B_2)  \fCenter \sim_2B_2$}
\UIC{$ \sim_2(A_2 \dor B_2)  \fCenter \sim_2B_2$}
\RL{\scriptsize \DU}
\UIC{$ \DU\sim_2(A_2 \dor B_2)  \fCenter \DU\sim_2B_2$}
\UIC{$ \DU\sim_2(A_2 \dor B_2)  \fCenter \du\sim_2B_2$}
\UIC{$ \du\sim_2(A_2 \dor B_2)  \fCenter \du\sim_2B_2$}
\BIC{$ \du\sim_2(A_2 \dor B_2) \UAND \du\sim_2(A_2 \dor B_2) \fCenter \du\sim_2A_2 \uand \du\sim_2B_2$}
\LL{\scriptsize C}
\UI $\du\sim_2A_2 \uand \du\sim_2B_2 \fCenter \du\sim_2(A_2 \dor B_2)$
\DP
\end{tabular}
 }
\end{center}

\newpage
$-(A \band B) \dashv\vdash - A  \bor - B \ \ \rightsquigarrow \ \ \du\sim_2(A_2 \dand B_2) \dashv\vdash \du\sim_2A_2 \uor \du\sim_2B_2$

\begin{center}
{
\begin{tabular}{c}
\!\!\!\!\!\!\!\!\!\!\!\!\!\!\!\!\!\!\!\!\!\!
\AX$A_2 \fCenter A_2 $
\LL{\scriptsize cont}
\UI$\ast_2 A_2 \fCenter \ast_2A_2 $
\UI$\ast_2 A_2 \fCenter \sim_2A_2 $
\RL{\scriptsize \DU}
\UI$\DU\ast_2 A_2 \fCenter \DU\sim_2A_2 $
\UI$\DU\ast_2 A_2 \fCenter \du\sim_2A_2 $
\RL{\scriptsize W}
\UI$\DU\ast_2 A_2 \fCenter \du\sim_2A_2 \UOR \du\sim_2B_2 $
\UI$\DU\ast_2 A_2 \fCenter \du\sim_2A_2 \uor \du\sim_2B_2 $
\LL{\scriptsize adj}
\UI$\ast_2 A_2 \fCenter \UD(\du\sim_2A_2 \uor \du\sim_2B_2) $
\LL{\scriptsize adj$\ast$}
\UIC{$\ast_2 \UD(\du\sim_2A_2 \uor \du\sim_2B_2)  \fCenter A_2$}
\AX$B_2 \fCenter B_2 $
\LL{\scriptsize cont}
\UI$\ast_2 B_2 \fCenter \ast_2B_2 $
\UI$\ast_2 B_2 \fCenter \sim_2B_2 $
\RL{\scriptsize \DU}
\UI$\DU\ast_2 B_2 \fCenter \DU\sim_2B_2 $
\UI$\DU\ast_2 B_2 \fCenter \du\sim_2B_2 $
\RL{\scriptsize W}
\UI$\DU\ast_2 B_2 \fCenter \du\sim_2B_2 \UOR \du\sim_2A_2 $
\RL{\scriptsize E}
\UI$\DU\ast_2 B_2 \fCenter \du\sim_2A_2 \UOR \du\sim_2B_2 $
\UI$\DU\ast_2 B_2 \fCenter \du\sim_2A_2 \uor \du\sim_2B_2 $
\LL{\scriptsize adj}
\UI$\ast_2 B_2 \fCenter \UD(\du\sim_2A_2 \uor \du\sim_2B_2) $
\LL{\scriptsize adj$\ast$}
\UIC{$\ast_2 \UD(\du\sim_2A_2 \uor \du\sim_2B_2)  \fCenter B_2$}
\BIC{$\ast_2 \UD(\du\sim_2A_2 \uor \du\sim_2B_2) \DAND \ast_2 \UD(\du\sim_2A_2 \uor \du\sim_2B_2)  \fCenter A_2 \dand B_2$}
\LL{\scriptsize C}
\UIC{$\ast_2 \UD(\du\sim_2A_2 \uor \du\sim_2B_2) \fCenter A_2 \dand B_2$}
\LL{\scriptsize adj$\ast$}
\UI$\ast_2(A_2 \dand B_2) \fCenter \UD(\du\sim_2A_2 \uor \du\sim_2B_2)$
\UI$\sim_2(A_2 \dand B_2) \fCenter \UD(\du\sim_2A_2 \uor \du\sim_2B_2)$
\LL{\scriptsize adj}
\UI$ \DU\sim_2(A_2 \dand B_2) \fCenter \du\sim_2A_2 \uor \du\sim_2B_2$
\UI$ \du\sim_2(A_2 \dand B_2) \fCenter \du\sim_2A_2 \uor \du\sim_2B_2$
\DP
\end{tabular}
 }
\end{center}

%$- A  \bor - B  \fCenter -(A \band B) \ \ \rightsquigarrow \ \ \du\sim_2A_2 \uor \du\sim_2B_2 \fCenter \du\sim_2(A_2 \dand B_2) $

\begin{center}
{
\begin{tabular}{c}
\AX$A_2 \fCenter A_2 $
\LL{\scriptsize cont}
\UI$*_2A_2 \fCenter *_2 A_2$
\UI$\sim_2A_2 \fCenter *_2 A_2$
\RL{\scriptsize \DU}
\UI$\DU\sim_2A_2 \fCenter \DU*_2 A_2$
\UI$\du\sim_2A_2 \fCenter \DU*_2 A_2$
\RL{\scriptsize adj}
\UI$\UD\du\sim_2A_2 \fCenter *_2 A_2$
\RL{\scriptsize adj$\ast$}
\UI$A_2  \fCenter *_2\UD\du\sim_2A_2$
\LL{\scriptsize W}
\UI$A_2 \DAND B_2  \fCenter *_2\UD\du\sim_2A_2$
\UI$A_2 \dand B_2  \fCenter *_2\UD\du\sim_2A_2$
\RL{\scriptsize adj$\ast$}
\UI$\UD\du\sim_2A_2  \fCenter *_2(A_2 \dand B_2)$
\UI$\UD\du\sim_2A_2  \fCenter \sim_2(A_2 \dand B_2)$
\RL{\scriptsize adj}
\UI$\du\sim_2A_2  \fCenter \DU\sim_2(A_2 \dand B_2)$
\UI$\du\sim_2A_2  \fCenter \du\sim_2(A_2 \dand B_2)$
\AX$B_2 \fCenter B_2 $
\LL{\scriptsize cont}
\UI$*_2B_2 \fCenter *_2 B_2$
\UI$\sim_2B_2 \fCenter *_2 B_2$
\RL{\scriptsize \DU}
\UI$\DU\sim_2B_2 \fCenter \DU*_2 B_2$
\UI$\du\sim_2B_2 \fCenter \DU*_2 B_2$
\RL{\scriptsize adj}
\UI$\UD\du\sim_2B_2 \fCenter *_2 B_2$
\RL{\scriptsize adj$\ast$}
\UI$B_2  \fCenter *_2\UD\du\sim_2B_2$
\LL{\scriptsize W}
\UI$B_2 \DAND A_2  \fCenter *_2\UD\du\sim_2B_2$
\LL{\scriptsize E}
\UI$A_2 \DAND B_2  \fCenter *_2\UD\du\sim_2B_2$
\UI$A_2 \dand B_2  \fCenter *_2\UD\du\sim_2B_2$
\RL{\scriptsize adj$\ast$}
\UI$\UD\du\sim_2B_2  \fCenter *_2(A_2 \dand B_2)$
\UI$\UD\du\sim_2B_2  \fCenter \sim_2(A_2 \dand B_2)$
\RL{\scriptsize adj}
\UI$\du\sim_2B_2  \fCenter \DU\sim_2(A_2 \dand B_2)$
\UI$\du\sim_2B_2  \fCenter \du\sim_2(A_2 \dand B_2)$
\BI$\du\sim_2A_2 \uor \du\sim_2B_2 \fCenter \du\sim_2(A_2 \dand B_2) \UOR \du\sim_2(A_2 \dand B_2)$
\RL{\scriptsize C}
\UI$\du\sim_2A_2 \uor \du\sim_2B_2 \fCenter \du\sim_2(A_2 \dand B_2) $
\DP
\end{tabular}
 }
\end{center}

\end{proof}

\subsection*{Conservativity}

\begin{proposition}
\label{prop:consequence preserved and reflected}
For all $\mathcal{L}$-formulas $A$ and $B$,
\[
\mathrm{if} \,\, HBL \models t_1(A) \leq_1 t_1(B) \,\,\textrm{then} \,\, A \vDash_\mathsf{B} B.
 \]
\end{proposition}
\begin{proof}
Assume that $A \nvDash_\mathsf{B} B$, then there exists a bilattice $\mathbb{B} \in \mathsf{B}$, such that $A^\mathsf{B} \in F_\mathsf{\bt}$ and $B^\mathsf{B} \notin F_\mathsf{\bt}$. By Proposition \ref{prop:from single to multi},  we have that there is an HBL $\mathbb{B}^+ = (\mathbb{L}_1,~\mathbb{L}_2,~\du,~\ud)$
%\footnote{In fact, $\mathrm{Reg}\mathbb{(B)} = \mathrm{Reg}\mathbb{(B)}$. Here, in order to distinct two different coordinates in product representation of the Bilattice $\mathbb{B}$, we use two different subscripts.}
, such that $t_1(A)^{\mathbb{L}_1} = \utop$ and $t_1(B)^{\mathbb{L}_1} \neq \utop$. Hence, $HBL \nvDash t_1(A) \leq_1 t_1(B)$. This argument also holds for HCBL.

\end{proof}

To argue that the calculus introduced in Section \ref{sec:disp} is conservative w.r.t.~BL (resp.~CBL), we follow the standard proof strategy discussed in \cite{GMPTZ,GKPLori}. 
Denote by  $\vdash_{\mathrm{BL}}$ (resp.~$\vdash_{\mathrm{CBL}}$) the %syntactic 
consequence relation defined by the calculus for $\mathrm{BL}$ (resp.~$\mathrm{CBL}$)
introduced in Section~\ref{sec:prel}, and by $\models_{\mathrm{HBL}}$ (resp.~$\models_{\mathrm{HCBL}}$)  the semantic consequence relation arising from (perfect) HBL
%-algebras 
(resp.~HCBL%-algebras
). 
We need to show that, for all formulas $A$ and $B$ of the original language of BL (resp.~CBL), if $t_1(A) \vdash t_1 (B)$ is a D.BL-derivable (resp.~D.CBL-derivable) sequent, then  $A\vdash_{\mathrm{BL}} B$ (resp.~ $A\vdash_{\mathrm{CBL}} B$). This  can be proved using  the following facts: (a) the rules of D.BL (resp.~D.CBL) are sound w.r.t.~perfect HBL-algebras  (resp.~HCBL-algebras);  (b) $\mathrm{BL}$ (resp.~CBL-algebras) is complete w.r.t.~$\mathsf{B}$ (resp.~$\mathsf{CB}$); and (c) $\mathsf{B}$ (resp.~$\mathsf{CB}$) are equivalently presented as (perfect) HBL-algebras (resp.~cf.~HCBL-algebras, Section \ref{Heterogeneous presentation}), so that the semantic consequence relations arising from each type of structures preserve and reflect the translation (cf.~Propositions \ref{prop1} and \ref{prop2}). 
Let then $A, B$ be formulas of  the original $\mathrm{BL}$ (resp.~$\mathrm{CBL}$)-language. If  $t_1(A) \vdash t_1(B)$ is a D.BL (resp.~D.CBL)-derivable sequent, then, by (a),  $t_1(A) \models_{\mathrm{HBL}} t_1(B)$ (resp.~$t_1(A) \models_{\mathrm{HCBL}} t_1(B)$). By (c) and Proposition \ref{prop:consequence preserved and reflected}
, this implies that $A\models_{\mathsf{B}} B$ (resp.~$A\models_{\mathsf{CB}} B$). By (b), this implies that $A\vdash_{\mathrm{BL}} B$ (resp.~$A\vdash_{\mathrm{CBL}} B$), as required.

\subsection*{Subformula property and cut elimination}

%In the present section, 
Let us briefly sketch the proof of cut elimination and subformula property for D.BL (resp.~D.CBL). As discussed earlier on, proper display calculi have been designed so that the cut elimination and subformula property  can  be inferred from a meta-theorem, following the strategy introduced by Belnap for display calculi \cite{Belnap1982}. The meta-theorem to which we will appeal for D.BL (resp.~D.CBL) was proved in \cite{TrendsXIII}.

All conditions in \cite[Theorem 4.1]{TrendsXIII} except $\mathrm{C}'_8$ are readily seen to be satisfied by inspection of the rules. Condition $\mathrm{C}'_8$ requires to check that reduction steps are available for every application of the cut rule in which both cut-formulas are principal, which either remove the original cut altogether or replace it by one or more cuts on formulas of strictly lower complexity.  In what follows, we only show  $\mathrm{C}'_8$ for the unary connectives $\sim$ and $\ud$ (the proof for $\du$ is analogous). The cases of lattice connectives are standard and they are omitted.

\paragraph*{$\mathsf{L}_i$-type connectives}
\begin{center}
%\vspace{-2.5ex}

{
\begin{tabular}{ccc}
\bottomAlignProof
\AXC{\ \ \ $\vdots$ \raisebox{1mm}{$\pi_1$}}
\noLine
\UI$X_i \fCenter \ast_i A_i$
\UI$X_i \fCenter \sim_i A_i$

\AXC{\ \ \ $\vdots$ \raisebox{1mm}{$\pi_2$}}
\noLine
\UI$\ast_i A_i \fCenter Y_i$
\UI$\sim_i A_i  \fCenter Y_i$
\BI$X_i \fCenter Y_i$
\DisplayProof

 & $\rightsquigarrow$ &

\bottomAlignProof
\AXC{\ \ \ $\vdots$ \raisebox{1mm}{$\pi_2$}}
\noLine
\UI$\ast_i A_i \fCenter Y_i$
\UI$\ast_i Y_i \fCenter A_i$
\AXC{\ \ \ $\vdots$ \raisebox{1mm}{$\pi_1$}}
\noLine
\UI$X_i \fCenter \ast_i A_i$
\UI$A_i \fCenter \ast_i X_i$
\BI$\ast_i Y_i  \fCenter \ast_i X_i$
\LL {\scriptsize cont}
\UI$X_i \fCenter Y_i$
\DisplayProof
 \\
\end{tabular}
 }
\end{center}

\paragraph*{Multi-type connectives}

\begin{center}
{
\begin{tabular}{ccc}
\bottomAlignProof
\AXC{\ \ \ $\vdots$ \raisebox{1mm}{$\pi_1$}}
\noLine
\UIC{$X_2 \fCenter \UD A_1$}
\UIC{$X_2 \fCenter \ud A_1$}
\AXC{\ \ \ $\vdots$ \raisebox{1mm}{$\pi_2$}}
\noLine
\UIC{$\UD A_1 \fCenter Y_2$}
\UIC{$\ud A_1 \fCenter Y_2$}
\BIC{$X_2 \fCenter Y_2$}
\DisplayProof

 & $\rightsquigarrow$ &

\bottomAlignProof
\AXC{\ \ \ $\vdots$ \raisebox{1mm}{$\pi_1$}}
\noLine
\UIC{$X_2 \fCenter \UD A_1$}
\UIC{$\DU X_2 \fCenter A_1$}
\AXC{\ \ \ $\vdots$ \raisebox{1mm}{$\pi_2$}}
\noLine
\UIC{$\UD A_1 \fCenter Y_2$}
\UIC{$A_1 \fCenter \DU Y_2$}
\BIC{$\DU X_2 \fCenter \DU Y_2$}
\LL {\scriptsize \DU}
\UIC{$X_2 \fCenter Y_2$}
\DisplayProof
 \\
\end{tabular}
 }
\end{center}

\section{Conclusions and future work}
\label{sec:conc}

The modular character of proper multi-type display calculi makes it possible to easily extend our formalism 
in order to axiomatize axiomatic extensions (e.g.~the logic of \emph{classical bilattices with conflation}~\cite[Definition 2.11]{arieli1996reasoning}) as well as language expansions of the basic bilattice logics treated
in the present paper. Expansions of bilattice logic have been extensively studied in the literature as early as 
in~\cite{arieli1996reasoning}, which introduces an implication enjoying the deduction-detachment theorem 
(see also~\cite{bou2013bilattices}). More recently, modal operators have been added to bilattice logics, 
motivated by potential applications to computer science and in particular verification of programs 
\cite{jung2013kripke,rivieccio2017four}; as well as dynamic modalities, motivated by applications
in the area of dynamic epistemic logic \cite{rivieccio2014algebraic, rivieccio2014bilattice}.

Yet more recently, bilattices with a negation not necessarily satisfying the involution law 
($\neg \neg a = a $) have been introduced with motivations of domain theory and topological duality
(see \cite{jakl2016bitopology}), and the study of the corresponding logics has been started
\cite{Rivieccio}. These logics are weaker than the one considered in the present paper, and so
adapting our display calculus formalism to them might prove a more challenging task 
(in particular, the translations introduced in Section~\ref{sec:prop}
may need to be redefined, as they rely on the maps $p$ and $n$ being lattice isomorphisms, which is no longer true in the non-involutive case).

\bibliographystyle{plain}
\bibliography{BilatticeLogicProperlyDisplayed}

\begin{thebibliography}{10}

\bibitem{arieli1996reasoning}
Ofer Arieli and Arnon Avron.
\newblock Reasoning with logical bilattices.
\newblock {\em Journal of Logic, Language and Information}, 5(1):25--63, 1996.

\bibitem{arieli1998value}
Ofer Arieli and Arnon Avron.
\newblock The value of the four values.
\newblock {\em Artificial Intelligence}, 102(1):97--141, 1998.

\bibitem{Belnap1982}
Nuel Belnap.
\newblock Display logic.
\newblock {\em J.\ Philos.\ Logic}, 11:375--417, 1982.

\bibitem{belnap1977computer}
Nuel~D. Belnap~Jr.
\newblock How a computer should think.
\newblock {\em Contemporary Aspects of Philosophy}, pages 30--56, 1977.

\bibitem{belnap1977useful}
Nuel~D. Belnap~Jr.
\newblock A useful four-valued logic.
\newblock In {\em Modern uses of multiple-valued logic}, pages 5--37. Springer,
  1977.

\bibitem{BGPTW}
Marta B\'{i}lkov\'{a}, Giuseppe Greco, Alessandra Palmigiano, Apostolos
  Tzimoulis, and Nachoem Wijnberg.
\newblock The logic of resources and capabilities.
\newblock {\em ArXiv preprint 1608.02222}, Submitted.

\bibitem{bou2011varieties}
F{\'e}lix Bou, Ramon Jansana, and Umberto Rivieccio.
\newblock Varieties of interlaced bilattices.
\newblock {\em Algebra universalis}, 66(1):115--141, 2011.

\bibitem{bou2010logic}
F{\'e}lix Bou and Umberto Rivieccio.
\newblock The logic of distributive bilattices.
\newblock {\em Logic Journal of the IGPL}, 19(1):183--216, 2010.

\bibitem{bou2013bilattices}
F{\'e}lix Bou and Umberto Rivieccio.
\newblock Bilattices with implications.
\newblock {\em Studia Logica}, 101(4):651--675, 2013.

\bibitem{CiabattoniRamanayake2016}
Agata Ciabattoni and Revantha Ramanayake.
\newblock Power and limits of structural display rules.
\newblock {\em ACM Transactions on Computational Logic (TOCL)}, 17(3):17, 2016.

\bibitem{CoGhPa14}
Willem Conradie, Silvio Ghilardi, and Alessandra Palmigiano.
\newblock {Unified Correspondence}.
\newblock In Alexandru Baltag and Sonja Smets, editors, {\em Johan van Benthem
  on Logic and Information Dynamics}, volume~5 of {\em Outstanding
  Contributions to Logic}, pages 933--975. Springer International Publishing,
  2014.

\bibitem{dunn1966algebra}
J.~Michael~Michael Dunn.
\newblock {\em The algebra of intensional logics}.
\newblock PhD thesis, Univesrity of Pittsburgh, 1966.

\bibitem{font1997belnap}
Josep~Maria Font.
\newblock Belnap's four-valued logic and {D}e {M}organ lattices.
\newblock {\em Logic Journal of IGPL}, 5(3):1--29, 1997.

\bibitem{PDL}
Sabine Frittella, Giuseppe Greco, Alexander Kurz, and Alessandra Palmigiano.
\newblock Multi-type display calculus for propositional dynamic logic.
\newblock {\em Journal of Logic and Computation}, 26(6):2067--2104, 2016.

\bibitem{TrendsXIII}
Sabine Frittella, Giuseppe Greco, Alexander Kurz, Alessandra Palmigiano, and
  Vlasta Sikimi\'{c}.
\newblock Multi-type sequent calculi.
\newblock {\em Proceedings Trends in Logic XIII, A. Indrzejczak, J. Kaczmarek,
  M. Zawidski eds}, 13:81--93, 2014.

\bibitem{Multitype}
Sabine Frittella, Giuseppe Greco, Alexander Kurz, Alessandra Palmigiano, and
  Vlasta Sikimi\'{c}.
\newblock Multi-type display calculus for dynamic epistemic logic.
\newblock {\em Journal of Logic and Computation}, 26(6):2067--2104, 2016.

\bibitem{Inquisitive}
Sabine Frittella, Giuseppe Greco, Alessandra Palmigiano, and Fan Yang.
\newblock A multi-type calculus for inquisitive logic.
\newblock In Jouko V{\"a}{\"a}n{\"a}nen, {\AA}sa Hirvonen, and Ruy de~Queiroz,
  editors, {\em Logic, Language, Information, and Computation: 23rd
  International Workshop, WoLLIC 2016, Puebla, Mexico, August 16-19th, 2016.
  Proceedings}, LNCS 9803, pages 215--233. Springer, 2016.

\bibitem{gehrke2001bounded}
Mai Gehrke and John Harding.
\newblock Bounded lattice expansions.
\newblock {\em Journal of Algebra}, 238(1):345--371, 2001.

\bibitem{gehrke1994bounded}
Mai Gehrke and Bjarni J{\'o}nsson.
\newblock Bounded distributive lattices with operators.
\newblock {\em Mathematica Japonica}, 40(2):207--215, 1994.

\bibitem{ginsberg1988multivalued}
Matthew~L. Ginsberg.
\newblock Multivalued logics: A uniform approach to reasoning in artificial
  intelligence.
\newblock {\em Computational intelligence}, 4(3):265--316, 1988.

\bibitem{GKPLori}
Giuseppe Greco, Alexander Kurz, and Alessandra Palmigiano.
\newblock Dynamic epistemic logic displayed.
\newblock In Huaxin Huang, Davide Grossi, and Olivier Roy, editors, {\em
  Proceedings of the 4th International Workshop on Logic, Rationality and
  Interaction (LORI-4)}, volume 8196 of {\em LNCS}, 2013.

\bibitem{SDM}
Giuseppe Greco, Fei Liang, Andrew Moshier, and Alessandra Palmigiano.
\newblock Multi-type display calculus for semi {D}e {M}organ logic.
\newblock {\em Proceedings of the 24th Workshop on Logic, Language, Information
  and Computation (WoLLIC)}, LNCS 10388:199--215, 2017.

\bibitem{GMPTZ}
Giuseppe Greco, Minghui Ma, Alessandra Palmigiano, Apostolos Tzimoulis, and
  Zhiguang Zhao.
\newblock Unified correspondence as a proof-theoretic tool.
\newblock {\em Journal of Logic and Computation}, 2016.

\bibitem{GrecoPalmigianoLatticeLogic}
Giuseppe Greco and Alessandra Palmigiano.
\newblock Lattice logic properly displayed.
\newblock {\em Proceedings of the 24th Workshop on Logic, Language, Information
  and Computation (WoLLIC). Proceedings}, LNCS 10388:153--169, 2017.

\bibitem{GP:linear}
Giuseppe Greco and Alessandra Palmigiano.
\newblock Linear logic properly displayed.
\newblock {\em arXiv:1611.04181}, Submitted.

\bibitem{jakl2016bitopology}
Tom{\'a}{\v{s}} Jakl, Achim Jung, and Ale{\v{s}} Pultr.
\newblock Bitopology and four-valued logic.
\newblock {\em Electronic Notes in Theoretical Computer Science}, 325:201--219,
  2016.

\bibitem{jung2013kripke}
Achim Jung and Umberto Rivieccio.
\newblock Kripke semantics for modal bilattice logic.
\newblock In {\em Proceedings of the 2013 28th Annual ACM/IEEE Symposium on
  Logic in Computer Science}, pages 438--447. IEEE Computer Society, 2013.

\bibitem{rivieccio2014algebraic}
Umberto Rivieccio.
\newblock Algebraic semantics for bilattice public announcement logic. a.
  indrzejczak, j. kaczmarek and m. zawidzki.
\newblock {\em Proceedings of Trends in Logic XIII (Lodz, Poland, 2-5 July
  2014), Lodz University Press}, pages 199--215, 2014.

\bibitem{rivieccio2014bilattice}
Umberto Rivieccio.
\newblock Bilattice public announcement logic.
\newblock {\em Advances in Modal Logic}, 10:459--477, 2014.

\bibitem{rivieccio2017four}
Umberto Rivieccio, Achim Jung, and Ramon Jansana.
\newblock Four-valued modal logic: Kripke semantics and duality.
\newblock {\em Journal of Logic and Computation}, 27(1):155--199, 2017.

\bibitem{Rivieccio}
Umberto Rivieccio, Paulo Maia, and Achim Jung.
\newblock Non-involutive twist-structures.
\newblock {\em Logic Journal of the IGPL}, Submitted.

\bibitem{Wansing1998}
Heinrich Wansing.
\newblock {\em Displaying Modal Logic}.
\newblock Kluwer, 1998.

\bibitem{wojcicki1988referential}
Ryszard W{\'o}jcicki.
\newblock {\em Theory of Logical Calculi}, chapter Referential semantics, pages
  341--401.
\newblock Springer, 1988.

\end{thebibliography}

\end{document}